\documentclass[reqno,12pt]{amsart}

\usepackage{latexsym}
\usepackage[margin=3cm]{geometry}
\usepackage{amscd}
\usepackage{amsthm}
\usepackage{marginnote}
\usepackage{subfigure,amsmath,amsfonts,amssymb,epsfig}
\usepackage{graphics}
\usepackage{xy}
\usepackage{mathrsfs}
\usepackage{tipa}
\usepackage{qtree}
\usepackage{forest}

\usepackage[latin1]{inputenc}

\usepackage{subfigure,amsmath}

\input xy
\xyoption{all}

\newcommand{\I}{\text{\sl\bf\textroundcap{\i}}}
\newcommand{\J}{\text{\sl\bf\textroundcap{\j}}}

\def\!{\mskip-\thinmuskip}
\def\,{\mskip\thinmuskip}
\def\;{\mskip\thickmuskip}

\newtheorem{theorem}{Theorem}[section]
\newtheorem{corollary}[theorem]{Corollary}
\newtheorem{proposition}[theorem]{Proposition}

\newtheorem{lemma}[theorem]{Lemma}

\newtheorem{definition}[theorem]{Definition}

\theoremstyle{remark}
\newtheorem{remark}[theorem]{Remark}
\newtheorem{example}[theorem]{Example}

\numberwithin{equation}{section}
\allowdisplaybreaks

\author[Michael J.\ Schlosser]{Michael J.\ Schlosser$^*$}
\address{Fakult\"at f\"ur Mathematik, Universit\"at Wien,
Oskar-Morgenstern-Platz~1, A-1090 Vienna, Austria}
\email{michael.schlosser@univie.ac.at}
\urladdr{http://www.mat.univie.ac.at/{\textasciitilde}schlosse}
\thanks{$^*$Partly supported by FWF Austrian Science Fund
grant F50-08 within the SFB
``Algorithmic and enumerative combinatorics''.}

\author[Meesue Yoo]{Meesue Yoo$^{**}$}
\address{Fakult\"at f\"ur Mathematik, Universit\"at Wien,
Oskar-Morgenstern-Platz~1, A-1090 Vienna, Austria}
\email{meesue.yoo@univie.ac.at}
\thanks{$^{**}$Fully supported by FWF Austrian Science Fund
grant F50-08 within the SFB
``Algorithmic and enumerative combinatorics''.}

\title{Elliptic rook and file numbers}



\newcommand{\ta}{\theta}
\newcommand{\C}{\mathbb C}

\newcommand{\N}{\mathbb N}

\begin{document}

\begin{abstract}
Utilizing elliptic weights, we construct an elliptic analogue of
rook numbers for Ferrers boards.
Our elliptic rook numbers generalize Garsia and Remmel's
$q$-rook numbers by two additional independent
 parameters $a$ and $b$, and a nome $p$.
These are shown to satisfy an elliptic extension of a  factorization
theorem which in the classical case was established
by Goldman, Joichi and White and later was extended to the
$q$-case by Garsia and Remmel. We obtain similar results
for our elliptic analogues of Garsia and Remmel's $q$-file numbers
for skyline boards.
We also provide an elliptic extension of the $j$-attacking model
introduced by Remmel and Wachs.
Various applications of our results include 
elliptic analogues of (generalized) Stirling numbers of the first and
second kind, Lah numbers, Abel numbers, and $r$-restricted versions
thereof.
\end{abstract}

\maketitle


\section{Introduction}\label{sec:intro}
The theory of rook numbers was introduced by Kaplansky and Riordan~\cite{KR}
in 1946, and since then it has been further studied and developed by 
many people. In 1975, Goldman, Joichi and White~\cite{GJW} proved the
following result for rook numbers on a Ferrers board
$B=B(b_1,\dots, b_n)\subset[n]\times\mathbb N$
(see Section~\ref{section:rook}
for the precise definition of rook numbers and of a Ferrers board):
\begin{equation}
\prod_{i=1}^n (z+b_i-i+1) =
\sum_{k=0}^n r_{n-k}(B)\,z(z-1)\dots(z-k+1).\label{eqn:rookthm}
\end{equation}
We refer to an identity of the form \eqref{eqn:rookthm} as a
{\em product formula} or {\em factorization theorem}.

In 1986 Garsia and Remmel~\cite{GR} established a $q$-analogue
of rook numbers on Ferrers boards by considering
a statistic involving {\em rook cancellation}.
Among other results, they were in particular able to extend
the product formula in \eqref{eqn:rookthm} to the $q$-case.
In 1991 Wachs and White~\cite{WaW} introduced a $(p,q)$-analogue
of rook numbers which was later studied in more detail
by Briggs and Remmel~\cite{BR}
and by Remmel and Wachs~\cite{RW}.
In 2001, Haglund and Remmel~\cite{HR} considered a suitably modified
statistic involving rook cancellation on {\em shifted} Ferrers boards.
In this way they were able to develop a rook theory for partial matchings
of the complete graph $K_{2n}$ of $2n$ vertices and 
in particular proved a product formula analogous to \eqref{eqn:rookthm}.

In the present work, we construct elliptic analogues of the rook numbers
on Ferrers boards by utilizing {\em elliptic weights}.
Our elliptic rook numbers generalize Garsia and Remmel's $q$-rook numbers
by two additional independent parameters $a$ and $b$, and a nome $p$.
They also contain the aforementioned $(p,q)$-rook numbers of \cite{BR,WaW}
as a special case.
We show that the elliptic rook numbers satisfy an elliptic extension
of \eqref{eqn:rookthm}.
Our elliptic rook numbers can be used to define elliptic analogues
of the Stirling numbers of the second kind, of the Lah numbers,
and of certain ``$r$-restricted'' refinements of these numbers,
by specializing Ferrers boards.

We also provide an elliptic extension
of the $\J$-attacking model which was considered by
Remmel and Wachs~\cite{RW}. As a consequence, we present elliptic
extensions of the generalized Stirling numbers of the first and the
second kinds.

Similarly, we construct elliptic analogues of the file numbers on
skyline boards which in the classical case and in the $q$-case
were first studied by Garsia and Remmel. We show that the elliptic
file numbers satisfy a factorization theorem, extending an analogous
result of Garsia and Remmel. Special cases of the elliptic
file numbers include an elliptic extension
of the unsigned Stirling numbers of the first kind, an elliptic
enumeration of labeled forests of rooted trees
which extends the work of Goldman and Haglund~\cite{GH} to the elliptic
setting, and again, elliptic extensions of $r$-restricted refinements
of these numbers.

At this point, we would like to explain our motivation for this work.
What is the reason for ``going elliptic''? People working in
enumerative combinatorics often encounter identities involving
hypergeometric series (or more generally, special functions).
On one hand, the theory of hypergeometric series
serves as a tool for solving combinatorial problems,
and on the other hand, combinatorial models can be used to prove
or explain hypergeometric series identities. This phenomenon similarly
also applies to other areas, such as algebra and geometry, or
mathematics and physics. Problems which lay in the interface of
two or more areas are often challenging and particularly interesting.
The development of tools which combine two different areas is promising
and may ultimately lead to a better understanding of both
respective theories.
Now, just as in many classical instances
where hypergeometric series emerge from problems in
enumerative combinatorics, $q$-hypergeometric or basic hypergeometric
series frequently emerge from problems in enumerative combinatorics
involving some $q$-statistics (which is a refined counting).
On the contrary, given specific basic hypergeometric series identities,
one can ask for suitable models where a combinatorial
explanation can be provided. Now this is not the end of the story.
There is in fact a natural hierarchy of hypergeometric series:
\emph{rational} (i.e. ``ordinary''),
\emph{trigonometric} (or ``$q$'', i.e., ``basic''),
and \emph{elliptic} (or ``$q,p$'', balanced and well-poised)
hypergeometric series. 
Not such a long time ago, people working in hypergeometric series
have realized that
the following three term relation of theta functions,
\begin{equation}
\ta(xy,x/y,uv,u/v;p)-\ta(xv,x/v,uy,u/y;p)
=\frac uy\,\ta(yv,y/v,xu,x/u;p)
\end{equation}
(see Section~\ref{sec:ellan} for the notation),
can be used as a key relation to build up a theory of identities
for series involving theta functions (see \cite{FT,Sp,SZ,Wa}
and the discussion in \cite[Chapter~11]{GR}),
analogous to the classical theories of
hypergeometric and of basic hypergeometric series.
This can be compared to the hierarchy 
of meromorphic solutions of the Yang--Baxter equation,
being rational, trigonometric, or elliptic, described in \cite{Ji}.
(Whereas trigonometric functions are periodic, elliptic functions
are doubly periodic. This cannot be pushed further, since by
Liouville's theorem, meromorphic functions on $\C$ with three
independent periods are constant.)
In the last three decades the theory of theta and elliptic
hypergeometric functions has been developed to a great extent
from various points of view (including integrable systems,
special functions, and biorthogonal functions).
But despite of their original appearance in
lattice models in statistical mechanics \cite{DJKMO},
elliptic hypergeometric series have not been studied much yet
from a combinatorial point of view. In \cite{Schl0}, one of us
enumerated lattice paths with respect to suitable elliptic weight functions.
This led to a combinatorial proof of the Frenkel--Turaev ${}_{10}V_9$
summation formula, a fundamental identity in the theory
of elliptic hypergeometric series. Further results from \cite{Schl0}
included the closed form elliptic enumeration of nonintersecting
lattice paths. Similar elliptic weights have also
subsequently been used in \cite{BGR} and in \cite{Be} to enumerate
dimers and lozenge tilings. In the quest of trying to better understand
the connection between combinatorics and elliptic hypergeometric series
it is just natural to look for suitable general combinatorial models
where elliptic weights can be utilized. The goal is to
obtain explicit results that generalize
the existing ones to the elliptic level, but which are still ``attractive''
(such as results involving closed form products).

Although we were successful in our aim to extend the classical rook theory
to the elliptic setting, we were somehow disappointed to
find that elliptic hypergeometric series summations did not come out
in this study. (See the discussion in Subsection~\ref{subsecehs}.)
This is probably inherent to the model. Already in the well-studied
$q$-case the only basic hypergeometric series identities that
arise in rook theory are those of Karlsson--Minton type
(see \cite{H0}, and again Subsection~\ref{subsecehs}).
Nevertheless, we are able for the first time to present a bunch
of elliptic extensions of special numbers (Stirling, Lah, etc.).
As these are new (and this territory opens up a new theory,
namely of {\em elliptic special numbers}),
we have put quite some attention to them in our exposition.
A reader who is not so much interested in all of these
new special numbers, is advised to mainly focus on the material
leading to the main result of this paper, the product formula
in Theorem~\ref{thm:elptprod} and to look at one or two specific
examples of applications. Besides that, a reading of
Section~\ref{sec:fp} in the end also serves to give an idea about
the ``big story''.

The paper is outlined as follows.
In Section~\ref{section:rook} we give a gentle introduction to rook theory,
state the product formula and a recursion for $q$-rook numbers.
In Section~\ref{sec:ellan}, after introducing elliptic analogues
of numbers and their properties, we generalize the results
in Section~\ref{section:rook} to the elliptic case and highlight
several special cases of interest. In Section \ref{sec:j}, we work out
an elliptic extension
of the $\J$-attacking rook model. We consider
elliptic file numbers in Section~\ref{sec:file} and  highlight several
special cases of interest there as well.
Lastly, we list some topics for future investigation in
Section~\ref{sec:fp}.

We would like to thank Ole Warnaar for reading an earlier version
of this manuscript and suggesting valuable improvements.


\section{Introduction to rook theory}\label{section:rook}

Let $\mathbb N$ denote the set of positive integers and  $\mathbb N_0$
the set of nonnegative integers.
We consider a \emph{board} to be a finite subset of the
$\mathbb N\times\mathbb N$ grid, and label the columns from left to
right with $1,2,3,\dots$, and the rows from bottom to top with
$1,2,3,\dots$. We let $(i,j)$ denote the cell in the
$i$-th column from the left and the $j$-th row from the bottom.
If a board has at most $n$ columns and $m$ rows,
we consider it as a subset of the $[n]\times[m]$ grid, where
$[n]=\{1,2,\dots, n\}$ and $[m]=\{1,2,\dots, m\}$.
For technical reasons, in our proofs, we sometimes find it convenient to
extend the $\mathbb N\times \mathbb N$ grid to $\mathbb N\times \mathbb Z$
where cells may have a zero or negative integer row index. 

Let $B(b_1, \dots, b_n)$ denote the set of cells 
\begin{displaymath}
B=B(b_1,\dots, b_n)=\{(i,j)~|~ 1\le i\le n,~ 1 \le j\le b_i\}.
\end{displaymath}
If a board $B$ can be represented by the set $B(b_1, \dots, b_n)$ for 
some nonnegative integer $b_i$'s, then the board $B$ is called 
a \emph{skyline board}. If in addition those $b_i$'s are nondecreasing, 
then the board $B=B(b_1,\dots, b_n)$ is called a \emph{Ferrers board}.

\setlength{\unitlength}{1.2pt}

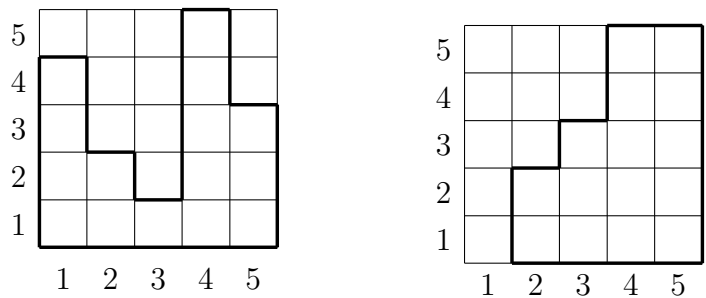
\begin{figure}[ht]
$$
\begin{picture}(85,85)(0,-10)
\multiput(10,0)(0,15){6}{\line(1,0){75}}
\multiput(10,0)(15,0){6}{\line(0,1){75}}
\put(1,64){5}
\put(1,49){4}
\put(1,34){3}
\put(1,19){2}
\put(1,4){1}
\put(15, -13){1}
\put(30,-13){2}
\put(45,-13){3}
\put(60,-13){4}
\put(75,-13){5}
\thicklines \linethickness{1.3pt}
\multiput(10,0)(0,15){1}{\line(0,1){60}}
\multiput(10,60)(0,15){1}{\line(1,0){15}}
\multiput(25,30)(0,15){1}{\line(0,1){30}}
\multiput(25,30)(0,15){1}{\line(1,0){15}}
\multiput(40,15)(0,15){1}{\line(0,1){15}}
\multiput(40,15)(0,15){1}{\line(1,0){15}}
\multiput(55,15)(0,15){1}{\line(0,1){60}}
\multiput(55,75)(0,15){1}{\line(1,0){15}}
\multiput(70,45)(0,15){1}{\line(0,1){30}}
\multiput(70,45)(0,15){1}{\line(1,0){15}}
\multiput(85,0)(0,15){1}{\line(0,1){45}}
\multiput(10,0)(0,15){1}{\line(1,0){75}}
\end{picture}
\qquad \qquad\quad
\begin{picture}(85,80)(0,-5)
\multiput(10,0)(0,15){6}{\line(1,0){75}}
\multiput(10,0)(15,0){6}{\line(0,1){75}}
\put(1,64){5}
\put(1,49){4}
\put(1,34){3}
\put(1,19){2}
\put(1,4){1}
\put(15,-10){1}
\put(30,-10){2}
\put(45,-10){3}
\put(60,-10){4}
\put(75,-10){5}
\thicklines \linethickness{1.3pt}
\multiput(25,0)(0,15){1}{\line(0,1){30}}
\multiput(25,30)(0,15){1}{\line(1,0){15}}
\multiput(40,30)(0,15){1}{\line(0,1){15}}
\multiput(40,45)(0,15){1}{\line(1,0){15}}
\multiput(55,45)(0,15){1}{\line(0,1){30}}
\multiput(55,75)(0,15){1}{\line(1,0){30}}
\multiput(25,0)(0,15){1}{\line(1,0){60}}
\multiput(85,0)(0,15){1}{\line(0,1){75}}
\end{picture}$$
\caption{A skyline board $B(4,2,1,5,3)$ and a Ferrers board $B(0,2,3,5,5)$.}
\end{figure}
We say that we
\emph{place $k$ nonattacking rooks} in $B$ by choosing a $k$-subset of
cells in $B$ such that no two elements have a common coordinate,
that is, no two rooks lie in the same row or in the same column.
Let $\mathcal{N}_{k}(B)$ denote the set of all nonattacking 
placements of $k$ rooks. The \emph{$k$-th rook number} of $B$ is defined by 
$r_k(B)=|\mathcal{N}_k (B)|$. 
To define the $q$-analogue of the rook numbers, we need the concept of
\emph{rook cancellation}. Given a rook placement $P\in\mathcal{N}_{k}(B)$,
a rook in $P$ cancels all the cells to the right in the same row and 
all the cells below it in the same column. Then
Garsia and Remmel \cite{GR} defined the $q$-analogue of the rook numbers
for Ferrers boards by
\begin{displaymath}
r_k (q;B)=\sum_{P\in \mathcal{N}_k(B)}q^{u_B (P)},
\end{displaymath}
where $q$ is an indeterminate and $u_B (P)$
counts the number of cells in $B$ which are neither
cancelled by rooks nor contain any rooks in a
$k$-rook placement $P$.
See Figure~\ref{fig2} for the set of cancelled cells (marked
by thick dots) of a particular placement of four rooks (marked by X's)
on the Ferrers board $B(0,2,3,5,5)$.

\begin{figure}[ht]\label{fig2}
$$\begin{picture}(85,80)(0,-5)
\multiput(10,0)(0,15){6}{\line(1,0){75}}
\multiput(10,0)(15,0){6}{\line(0,1){75}}
\put(1,64){5}
\put(1,49){4}
\put(1,34){3}
\put(1,19){2}
\put(1,4){1}
\put(15, -10){1}
\put(30,-10){2}
\put(45,-10){3}
\put(60,-10){4}
\put(75,-10){5}
\thicklines  \linethickness{1.3pt}
\multiput(25,0)(0,15){1}{\line(0,1){30}}
\multiput(25,30)(0,15){1}{\line(1,0){15}}
\multiput(40,30)(0,15){1}{\line(0,1){15}}
\multiput(40,45)(0,15){1}{\line(1,0){15}}
\multiput(55,45)(0,15){1}{\line(0,1){30}}
\multiput(55,75)(0,15){1}{\line(1,0){30}}
\multiput(25,0)(0,15){1}{\line(1,0){60}}
\multiput(85,0)(0,15){1}{\line(0,1){75}}
\multiput(27,17)(0,15){1}{\line(1,1){11}}
\multiput(27,28)(0,15){1}{\line(1,-1){11}}
\multiput(42,2)(0,15){1}{\line(1,1){11}}
\multiput(42,13)(0,15){1}{\line(1,-1){11}}
\multiput(57,47)(0,15){1}{\line(1,1){11}}
\multiput(57,58)(0,15){1}{\line(1,-1){11}}
\multiput(72,32)(0,15){1}{\line(1,1){11}}
\multiput(72,43)(0,15){1}{\line(1,-1){11}}
\put(30,5){$\bullet$}
\put(60,5){$\bullet$}
\put(75,5){$\bullet$}
\put(45,20){$\bullet$}
\put(60,20){$\bullet$}
\put(75,20){$\bullet$}
\put(60,35){$\bullet$}
\put(75,50){$\bullet$}
\end{picture}$$
\caption{A rook cancellation in $B(0,2,3,5,5)$.}
\end{figure}
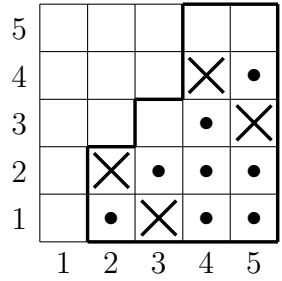

Note that for $B=[n]\times [n]$, a nonattacking rook placement of $n$ rooks
in $B$ corresponds to a permutation of $1,2,\dots, n$.
By considering all the rook placements corresponding to the permutations of $n$ 
numbers, it is not hard to see that 
\begin{equation}\label{mah}
r_n (q;[n]\times [n])=[n]_q!.
\end{equation}
Here the $q$-falling factorial and $q$-factorial are defined by
\begin{equation*}
[n]_q\!\downarrow_k=[n]_q [n-1]_q \dots [n-k+1]_q,
\qquad\text{and}\qquad
[n]_q!=[n]_q\!\downarrow_n
\end{equation*}
with $[n]_q\!\downarrow_0=1$,
respectively, where
\begin{equation*}
[n]_q =\frac{1-q^n}{1-q}
\end{equation*}
is the $q$-number of $n$.
From \eqref{mah} it can be concluded that the rook numbers
satisfy the {\em Mahonian} property
(first defined in \cite{F}), i.e.\ the two statistics, the number of
uncancelled cells of $n$ nonattacking rooks on an $[n]\times [n]$ board
and the number of inversions
of permutations of $n$ elements, have the same distribution.

For a given Ferrers board $B\subset[n]\times\mathbb N$,
let us denote by $B^\infty\subset[n]\times\mathbb Z$ the Ferrers board
obtained by appending below $B$ the infinite board of width $n$. 
For convenience, we denote by $\mathfrak{g}$ the line separating $B$
from the rest of $B^\infty$ and refer to it as the \emph{ground}.
For a rook placement $P$ in $B^\infty$, we let $\text{max}(P)$ denote
the number of rows below the ground in which the lowest rook of $P$
is located. Let $\text{max}(P)=0$ if there are no rooks below
$\mathfrak{g}$ in $P$. Garsia and Remmel~\cite[Equation~(I.11)]{GR}
showed the following identity.

\begin{proposition}\cite[Equation~(I.11)]{GR}\label{prop:q1}
For any Ferrers board $B=B(b_1,\dots, b_n)$,  
\begin{equation}
\frac{1}{1-z}\sum_{P\in \mathcal{N}_n(B^\infty)}z^{\text{max}(P)}q^{u_B (P)}
=\sum_{k\ge 0}z^k [k+b_1]_q [k+b_2 -1]_q \cdots [k+b_n -n+1]_q.
\end{equation}
\end{proposition}

Using this result, Garsia and Remmel~\cite[Equation~(1.3)]{GR}
proved the following factorization theorem for $q$-rook numbers
on Ferrers boards which extends the result of Goldman, Joichi and White
in \eqref{eqn:rookthm}.

\begin{proposition}\cite[Equation~(1.3)]{GR}\label{thm:qrookthm}
Let $B=B(b_1,\dots, b_n)\subset [n]\times \mathbb N$ be
a Ferrers board. Then 
\begin{equation}
\prod_{i=1}^n [z+b_i-i+1]_q =
\sum_{k=0}^n r_{n-k}(q;B)\,[z]_q\!\downarrow_k.\label{eqn:qrookthm}
\end{equation} 
\end{proposition}
We recover \eqref{eqn:rookthm} when $q\to 1$. 
By distinguishing whether there is a rook
in the last column or not, Garsia and Remmel also showed the following recursion
\cite[Theorem~1.1]{GR}. 
\begin{proposition}\cite[Theorem~1.1]{GR}\label{thm:recur}
Let $B$ be a Ferrers board of height at most $m$ and let $B\cup m$
denote the board obtained by adding a column of length $m$ to $B$.
Then for any nonnegative integer $k$, we have 
\begin{equation}
r_k(q;B\cup m)=q^{m-k}r_k (q;B)+[m-k+1]_q\,r_{k-1}(q;B).
\end{equation}
\end{proposition}

Now we are ready to turn to the elliptic setting.


\section{Elliptic analogues}\label{sec:ellan}

A function is defined to be {\em elliptic} if it is meromorphic and
doubly periodic. It is well known (cf.\ e.g.\ \cite{W}) that
elliptic functions can be built from quotients of
theta functions.

Define a \emph{modified Jacobi theta function} with argument $x$ and nome $p$ by
$$\theta(x;p)= \prod_{j\ge 0}((1-p^j x)(1- p^{j+1}/x)),\qquad
\theta(x_1,\dots, x_m;p)=\prod_{k=1}^m \theta(x_k;p),$$
where $x,x_1,\dots, x_m\ne 0$, $|p|<1$. 
Further, we define the {\em theta shifted factorial}
(or {\em $q,p$-shifted factorial}) by
\begin{equation*}
(a;q,p)_n = \begin{cases}
\prod^{n-1}_{k=0} \theta (aq^k;p),&\quad n = 1, 2, \ldots\,,\cr
1,&\quad n = 0,\cr
1/\prod^{-n-1}_{k=0} \theta (aq^{n+k};p),&\quad n = -1, -2, \ldots,
\end{cases}
\end{equation*}
together with
\begin{equation*}
(a_1, a_2, \ldots, a_m;q, p)_n = \prod^m_{k=1} (a_k;q,p)_n,
\end{equation*}
for compact notation.
For $p=0$ we have  $\theta (x;0) = 1-x$ and, hence, $(a;q, 0)_n = (a;q)_n$
is a {\em $q$-shifted factorial} in base $q$ (see \cite{GRhyp} for
classical $q$-series notation, \cite[Chapter~11]{GRhyp} treats the
elliptic case).
The parameters $q$ and $p$
in $(a;q,p)_n$ are called the {\em base} and {\em nome}, respectively. 

The modified Jacobi theta functions satisfy the following
basic properties which are essential in the theory of elliptic
hypergeometric series:
\begin{subequations}
\begin{equation}\label{tif}
\ta(x;p)=-x\,\ta(1/x;p),
\end{equation}
\begin{equation}\label{p1id}
\ta(px;p)=-\frac 1x\,\ta(x;p),
\end{equation}
and the {\em addition formula}
\begin{equation}\label{addf}
\ta(xy,x/y,uv,u/v;p)-\ta(xv,x/v,uy,u/y;p)
=\frac uy\,\ta(yv,y/v,xu,x/u;p)
\end{equation}
\end{subequations}
(cf.\ \cite[p.~451, Example 5]{WhW}).

As a matter of fact, the three-term relation in \eqref{addf},
containing four variables
and four factors of theta functions in each term, is the ``smallest''
addition formula connecting products of theta functions with general
arguments.
Note that in the theta function $\theta(x;p)$ we cannot let $x\to 0$
(unless we first let $p\to 0$) for $x$ is a pole of infinite order.
This is the reason why elliptic analogues of
$q$-series identities usually contain many parameters.

The elliptic identities we shall consider all involve terms
which are elliptic (with the same periods) in all of its parameters 
(see e.g.\ Remark~\ref{rem:totalell}).
Spiridonov~\cite{Sp} refers to such multivariate functions as
{\em totally elliptic}, and they are by nature
{\em well-poised} and {\em balanced} (see also \cite[Chapter~11]{GRhyp}).

Inspired by earlier work of the first author regarding
weighted lattice paths and elliptic binomial coefficients
\cite{Schl0,Schl1}, we now define the {\em elliptic weights}
$w_{a,b;q,p}(k)$ and $W_{a,b;q,p}(k)$,
depending on two independent parameters $a$ and $b$, base $q$,
nome $p$, and integer parameter $k$ by
\begin{subequations}\label{def:elpwt}
\begin{align}
w_{a,b;q,p}(k)&=\frac{\theta(aq^{2k+1},bq^{k},aq^{k-2}/b;p)}
{\theta(aq^{2k-1},bq^{k+2},aq^k /b;p)}q,\label{def:smallelpwt}\\\intertext{and}
W_{a,b;q,p}(k)&=
\frac{\theta(aq^{1+2k},bq,bq^2,aq^{-1}/b,a/b;p)}
{\theta(aq,bq^{k+1},bq^{k+2},aq^{k-1}/b, aq^k /b;p)}q^k,\label{def:bigelpwt}
\end{align}
respectively. Observe that if $k$ is a positive integer, Equations
\eqref{def:smallelpwt} and \eqref{def:bigelpwt} imply that
\begin{equation}
W_{a,b;q,p}(k)=\prod_{j=1}^k w_{a,b;q,p}(j).
\end{equation}
\end{subequations}
We refer to the $w_{a,b;q,p}(k)$ as {\em small} weights
and to the $W_{a,b;q,p}(k)$ as {\em big} weights. 
Note that the weights $w_{a,b;q,p}(k)$ and $W_{a,b;q,p}(k)$ also can be
defined for arbitrary (complex) $k$ which is clear from the definition.

Observe that
\begin{subequations}
\begin{equation}\label{wshift}
w_{a,b;q,p}(k+n)=w_{aq^{2k},bq^k;q,p}(n),
\end{equation}
and 
\begin{equation}\label{Wshift}
W_{a,b;q,p}(k+n)=W_{a,b;q,p}(k)\,
W_{aq^{2k},bq^k;q,p}(n),
\end{equation}
\end{subequations}
for all $k$ and $n$,
which are elementary identities we frequently make use of.

\begin{remark}\label{rem:totalell}
The small weight $w_{a,b;q,p}(k)$ (and so the big one) is indeed
elliptic in its parameters (i.e., totally elliptic).
If we write $q=e^{2\pi i\sigma}$,
$p=e^{2\pi i\tau}$, $a=q^\alpha$ and $b=q^\beta$ with complex $\sigma$,
$\tau$, $\alpha$, $\beta$ and $k$, then the small weight
$w_{a,b;q,p}(k)$ is clearly periodic in $\alpha$ with period $\sigma^{-1}$.
A simple computation involving \eqref{p1id} further shows that
$w_{a,b;q,p}(k)$ is also periodic in $\alpha$ with period $\tau\sigma^{-1}$.
The same applies to $w_{a,b;q,p}(k)$ as a function in $\beta$ (or $k$)
with the same two periods $\sigma^{-1}$ and $\tau\sigma^{-1}$.
\end{remark}

\begin{remark}
For $p\to 0$, the small and big weights reduce to
\begin{subequations}\label{def:qwt}
\begin{align}
w_{a,b;q}(k)&=\frac{(1-aq^{2k+1})(1-bq^{k})(1-aq^{k-2}/b)}
{(1-aq^{2k-1})(1-bq^{k+2})(1-aq^k /b)}q,\label{def:smallqwt}\\\intertext{and}
W_{a,b;q}(k)&=
\frac{(1-aq^{1+2k})(1-bq)(1-bq^2)(1-aq^{-1}/b)(1-a/b)}
{(1-aq)(1-bq^{k+1})(1-bq^{k+2})(1-aq^{k-1}/b)(1-aq^k /b)}q^k,\label{def:bigqwt}
\end{align}
\end{subequations}
respectively.
In the $a,b;q$-weights in \eqref{def:qwt},
we may let $b\to 0$ (or $b\to\infty$) to obtain ``$a,0;q$-weights'',
or in short, ``$a;q$-weights'':
\begin{equation}\label{def:aqwt}
w_{a;q}(k)=\frac{(1-aq^{2k+1})}{(1-aq^{2k-1})}q^{-1},\qquad\text{and}\qquad
W_{a;q}(k)=\frac{(1-aq^{1+2k})}{(1-aq)}q^{-k}.
\end{equation}
Note that by writing $q=e^{ix}$ and $a=e^{i(2c+1)x}$, $c\in\mathbb N$,
the $a;q$-weights
can be written as quotients of Chebyshev polynomials
of the second kind.

Also, in \eqref{def:qwt}, we may let $a\to 0$ (or $a\to\infty$) to obtain
``$0,b;q$-weights''.
Importantly, if in \eqref{def:qwt} we first let $b\to 0$ and then
$a\to\infty$ (or, equivalently, first let $a\to 0$ and then $b\to 0$),
we obtain the familiar $q$-weights
\begin{equation}
w_q(k)=q\qquad\text{and}\qquad W_q(k)=q^k,
\end{equation}
respectively.
\end{remark}

Next, for a variable $z$, 
we define an \emph{elliptic number} of $z$ by
\begin{equation}\label{elln}
[z]_{a,b;q,p}=\frac{\theta(q^z, aq^z, bq^2, a/b;p)}
{\theta(q,aq,bq^{z+1},aq^{z-1}/b;p)}.
\end{equation}

Using the addition formula for theta functions~\eqref{addf},
it is not difficult to verify that the thus defined elliptic numbers satisfy
\begin{subequations}
\begin{equation}\label{recelln}
[z]_{a,b;q,p} = [z-1]_{a,b;q,p}+W_{a,b;q,p}(z-1).
\end{equation}
In case $z=n$ is a nonnegative integer, \eqref{recelln} constitutes a
recursion which, together with $W_{a,b;q,p}(0)=1$,
uniquely defines any elliptic number
$[n]_{a,b;q,p}$.
More generally, by \eqref{addf} we have the following useful identity
\begin{equation}\label{recellny}
[z]_{a,b;q,p} = [y]_{a,b;q,p}+W_{a,b;q,p}(y)[z-y]_{aq^{2y},bq^y;q,p}
\end{equation}
\end{subequations}
which reduces to \eqref{recelln} for $y=z-1$.

\begin{remark}
In \cite{Schl1}, the first author, 
in analogy to the $q$-binomial coefficients 
\begin{equation*}
\begin{bmatrix}n\\k\end{bmatrix}_q
:=\frac{(q^{1+k};q)_{n-k}}
{(q;q)_{n-k}}=\frac{[n]_q!}{[k]_q![n-k]_q!},
\end{equation*}
defined the elliptic binomial coefficients 
\begin{equation}\label{ellbin}
\begin{bmatrix}n\\k\end{bmatrix}_{a,b;q,p}:=
\frac{(q^{1+k},aq^{1+k},bq^{1+k},aq^{1-k}/b;q,p)_{n-k}}
{(q,aq,bq^{1+2k},aq/b;q,p)_{n-k}}.
\end{equation}
In \cite{Schl1} the elliptic binomial coefficients in \eqref{ellbin}
were shown to satisfy an elliptic binomial theorem involving
``elliptic commuting'' variables.
They were also shown to satisfy a nice recursion, namely
\begin{subequations}\label{wbineq}
\begin{align}\label{recu}
&\begin{bmatrix}0\\0\end{bmatrix}_{a,b;q,p}=1,\qquad
\begin{bmatrix}n\\k\end{bmatrix}_{a,b;q,p}=0
\qquad\text{for\/ $n\in\N_0$, and\/
$k\in-\N$ or $k>n$},\\
\intertext{and}
\label{recw}
&
\begin{bmatrix}n+1\\k\end{bmatrix}_{a,b;q,p}=
\begin{bmatrix}n\\k\end{bmatrix}_{a,b;q,p}
+\begin{bmatrix}n\\k-1\end{bmatrix}_{a,b;q,p}
\,W_{aq^{k-1},bq^{2k-2};q,p}(n+1-k)
\qquad \text{for $n,k\in\N_0$}.
\end{align}
\end{subequations}
The recurrence in \eqref{recw} is a consequence of the
addition formula \eqref{addf}.

On the combinatorial side, the elliptic binomial coefficient
in \eqref{ellbin}
can be interpreted in terms of weighted lattice paths in $\mathbb Z^2$ 
(see \cite{Schl0}).
In fact, \eqref{ellbin} is the area generating function for
paths starting in $(0,0)$ and ending in $(k,n-k)$ composed of unit steps going
north or east only, when the weight of each cell
(with north-east corner $(s,t)$) ``covered'' by the path
is defined to be $w_{aq^{s-1},bq^{2s-2};q,p}(t)$.
Then it can be shown that the sum of weighted areas below the paths 
satisfies the same recursion \eqref{recw} by distinguishing the 
last step of the path which is either vertical or horizontal.
The elliptic number $[n]_{a,b;q,p}$ is nothing but a short-hand notation for
\begin{equation*}
[n]_{a,b;q,p}=\begin{bmatrix}n\\1\end{bmatrix}_{a,b;q,p},
\end{equation*}
the weighted enumeration of all paths starting in 
$(0,0)$ and ending in $(1,n-1)$.
Note that the elliptic binomial coefficients are in general \emph{not} symmetric with respect to replacing
$k$ by $n-k$. However, the $a;q$-binomial coefficients
\begin{equation}\label{aqbin}
\begin{bmatrix}n\\k\end{bmatrix}_{a;q}:=
\frac{(q^{1+k},aq^{1+k};q)_{n-k}}
{(q,aq;q)_{n-k}}q^{k(k-n)}
\end{equation}
obtained from \eqref{ellbin} by formally letting $p\to 0$ followed by
$b\to 0$, are symmetric, i.e., they satisfy
\begin{equation*}
\begin{bmatrix}n\\k\end{bmatrix}_{a;q}=
\begin{bmatrix}n\\n-k\end{bmatrix}_{a;q},
\end{equation*}
and this is a reason for the $a;q$-case being special (see e.g.\
Prop.~\ref{prop:aqrect}). The
$q$-binomial coefficient is recovered by letting $a\to\infty$.

\end{remark}

Now we are ready to develop an elliptic analogue of the $q$-rook theory.
We employ the same rook cancellation as Garsia and Remmel
considered in the $q$-case,
i.e., a rook cancels all the cells to the right and below of it.
However, as already indicated, our definition of an elliptic-weighted
rook number requires a refined statistic depending on the
specific locations of the uncancelled cells.

\begin{definition}
Given a Ferrers board $B=B(b_1,\dots, b_n)$,
we define the elliptic analogue of the
$k$-th rook number by
\begin{subequations}
\begin{equation}
r_k(a,b;q,p;B)=\sum_{P\in \mathcal{N}_k(B)}\text{wt}(P),\label{def:elptrook}
\end{equation}
with 
\begin{equation}\label{def:wt}
\text{wt}(P)=\prod_{(i,j)\in \text{U}_B(P)}
w_{a,b;q,p}\big(i-j-r_{(i,j)}(P)\big),
\end{equation}
\end{subequations}
where the elliptic weight $w_{a,b;q,p}(l)$ of an integer $l$ is
defined in \eqref{def:smallelpwt}, $\text{U}_B(P)$ is the set
of cells in $B$ which are neither
cancelled by rooks nor contain any rooks of $P$, and $r_{(i,j)}(P)$
is the number of rooks in $P$ which are in the north-west region of $(i,j)$. 
\end{definition}

\begin{example} Consider a Ferrers board $B=B(3,3,3)$ and let
$P$ be the placement of two rooks in $(1,3)$ and $(3,1)$ in $B$. 
\begin{figure}[ht]
$$\begin{picture}(55,50)(0,-5)
\multiput(10,0)(0,15){4}{\line(1,0){45}}
\multiput(10,0)(15,0){4}{\line(0,1){45}}
\multiput(12,32)(0,15){1}{\line(1,1){11}}
\multiput(12,43)(0,15){1}{\line(1,-1){11}}
\multiput(42,2)(0,15){1}{\line(1,1){11}}
\multiput(42,13)(0,15){1}{\line(1,-1){11}}
\put(1,34){3}
\put(1,19){2}
\put(1,4){1}
\put(15, -10){1}
\put(30,-10){2}
\put(45,-10){3}
\put(15,5){$\bullet$}
\put(15,20){$\bullet$}
\put(30,35){$\bullet$}
\put(45,35){$\bullet$}
\end{picture}$$
\caption{A rook placement in $(1,3)$ and $(3,1)$.}
\end{figure}
Then the set of uncancelled cells $U_B(P)$ is 
$\{ (2,1), (2,2), (3,2)\}$. Note that for all the uncancelled cells
$(i,j)\in U_B(P)$, $r_{(i,j)}(P)=1$ due to the rook in $(1,3)$.
Then $wt(P)$ is 
\begin{align*}
wt(P)
&= w_{a,b;q,p}(2-1-1)\cdot  w_{a,b;q,p}(2-2-1)
\cdot w_{a,b;q,p}(3-2-1)\\
&= \frac{\theta(aq^{-1},bq^{-1},aq^{-3}/b;p)}{\theta(aq^{-3},bq,aq^{-1}/b;p)}\cdot
\frac{\theta(aq,b,aq^{-2}/b;p)^2}{\theta(aq^{-1},bq^{2},a/b;p)^2}q^3.
\end{align*}

\setlength{\unitlength}{1.1pt}
If we place three rooks in all possible ways in $B$ and compute
$r_3(a,b;q,p;B)$, then 
\begin{figure}[ht]
$$\begin{array}{cccccc}
\begin{picture}(55,50)(0,-5)
\multiput(10,0)(0,15){4}{\line(1,0){45}}
\multiput(10,0)(15,0){4}{\line(0,1){45}}
\multiput(12,32)(0,15){1}{\line(1,1){11}}
\multiput(12,43)(0,15){1}{\line(1,-1){11}}
\multiput(27,17)(0,15){1}{\line(1,1){11}}
\multiput(27,28)(0,15){1}{\line(1,-1){11}}
\multiput(42,2)(0,15){1}{\line(1,1){11}}
\multiput(42,13)(0,15){1}{\line(1,-1){11}}
\put(1,34){3}
\put(1,19){2}
\put(1,4){1}
\put(15, -10){1}
\put(30,-10){2}
\put(45,-10){3}
\put(15,5){$\bullet$}
\put(30,5){$\bullet$}
\put(15,20){$\bullet$}
\put(45,20){$\bullet$}
\put(30,35){$\bullet$}
\put(45,35){$\bullet$}
\end{picture}
&
\begin{picture}(55,50)(0,-5)
\multiput(10,0)(0,15){4}{\line(1,0){45}}
\multiput(10,0)(15,0){4}{\line(0,1){45}}
\multiput(12,32)(0,15){1}{\line(1,1){11}}
\multiput(12,43)(0,15){1}{\line(1,-1){11}}
\multiput(42,17)(0,15){1}{\line(1,1){11}}
\multiput(42,28)(0,15){1}{\line(1,-1){11}}
\multiput(27,2)(0,15){1}{\line(1,1){11}}
\multiput(27,13)(0,15){1}{\line(1,-1){11}}
\put(1,34){3}
\put(1,19){2}
\put(1,4){1}
\put(15, -10){1}
\put(30,-10){2}
\put(45,-10){3}
\put(15,5){$\bullet$}
\put(45,5){$\bullet$}
\put(15,20){$\bullet$}
\put(30,35){$\bullet$}
\put(45,35){$\bullet$}
\end{picture}
&
\begin{picture}(55,50)(0,-5)
\multiput(10,0)(0,15){4}{\line(1,0){45}}
\multiput(10,0)(15,0){4}{\line(0,1){45}}
\multiput(27,32)(0,15){1}{\line(1,1){11}}
\multiput(27,43)(0,15){1}{\line(1,-1){11}}
\multiput(12,17)(0,15){1}{\line(1,1){11}}
\multiput(12,28)(0,15){1}{\line(1,-1){11}}
\multiput(42,2)(0,15){1}{\line(1,1){11}}
\multiput(42,13)(0,15){1}{\line(1,-1){11}}
\put(1,34){3}
\put(1,19){2}
\put(1,4){1}
\put(15, -10){1}
\put(30,-10){2}
\put(45,-10){3}
\put(15,5){$\bullet$}
\put(30,5){$\bullet$}
\put(30,20){$\bullet$}
\put(45,20){$\bullet$}
\put(45,35){$\bullet$}
\end{picture}
&
\begin{picture}(55,50)(0,-5)
\multiput(10,0)(0,15){4}{\line(1,0){45}}
\multiput(10,0)(15,0){4}{\line(0,1){45}}
\multiput(42,32)(0,15){1}{\line(1,1){11}}
\multiput(42,43)(0,15){1}{\line(1,-1){11}}
\multiput(12,17)(0,15){1}{\line(1,1){11}}
\multiput(12,28)(0,15){1}{\line(1,-1){11}}
\multiput(27,2)(0,15){1}{\line(1,1){11}}
\multiput(27,13)(0,15){1}{\line(1,-1){11}}
\put(1,34){3}
\put(1,19){2}
\put(1,4){1}
\put(15, -10){1}
\put(30,-10){2}
\put(45,-10){3}
\put(45,20){$\bullet$}
\put(30,20){$\bullet$}
\put(15,5){$\bullet$}
\put(45,5){$\bullet$}
\end{picture}
&
\begin{picture}(55,50)(0,-5)
\multiput(10,0)(0,15){4}{\line(1,0){45}}
\multiput(10,0)(15,0){4}{\line(0,1){45}}
\multiput(27,32)(0,15){1}{\line(1,1){11}}
\multiput(27,43)(0,15){1}{\line(1,-1){11}}
\multiput(42,17)(0,15){1}{\line(1,1){11}}
\multiput(42,28)(0,15){1}{\line(1,-1){11}}
\multiput(12,2)(0,15){1}{\line(1,1){11}}
\multiput(12,13)(0,15){1}{\line(1,-1){11}}
\put(1,34){3}
\put(1,19){2}
\put(1,4){1}
\put(15, -10){1}
\put(30,-10){2}
\put(45,-10){3}
\put(30,5){$\bullet$}
\put(45,5){$\bullet$}
\put(30,20){$\bullet$}
\put(45,35){$\bullet$}
\end{picture}
&
\begin{picture}(55,50)(0,-5)
\multiput(10,0)(0,15){4}{\line(1,0){45}}
\multiput(10,0)(15,0){4}{\line(0,1){45}}
\multiput(42,32)(0,15){1}{\line(1,1){11}}
\multiput(42,43)(0,15){1}{\line(1,-1){11}}
\multiput(27,17)(0,15){1}{\line(1,1){11}}
\multiput(27,28)(0,15){1}{\line(1,-1){11}}
\multiput(12,2)(0,15){1}{\line(1,1){11}}
\multiput(12,13)(0,15){1}{\line(1,-1){11}}
\put(1,34){3}
\put(1,19){2}
\put(1,4){1}
\put(15, -10){1}
\put(30,-10){2}
\put(45,-10){3}
\put(30,5){$\bullet$}
\put(45,20){$\bullet$}
\put(45,5){$\bullet$}
\end{picture}
\end{array}$$
\caption{Rook placements of three rooks in $[3]\times [3]$.}
\end{figure}
\begin{align*}
&r_3(a,b;q,p;B)\\
&= 1+w_{a,b;q,p}(-1)+w_{a,b;q,p}(-2)+
2\cdot w_{a,b;q,p}(-2)w_{a,b;q,p}(-1)+
w_{a,b;q,p}(-2)w_{a,b;q,p}(-1)^2\\
&= (1+w_{aq^{-6},bq^{-3};q,p}(1)+w_{aq^{-6},bq^{-3};q,p}(1)w_{aq^{-6},bq^{-3};q,p}(2))
(1+w_{aq^{-6},bq^{-3};q,p}(2))\\
&=(1+W_{aq^{-6},bq^{-3};q,p}(1)+W_{aq^{-6},q^{-3}b;q,p}(2))(1+W_{aq^{-4},bq^{-2};q,p}(1))\\
&=[3]_{aq^{-6},bq^{-3};q,p}[2]_{aq^{-4},bq^{-2};q,p},
\end{align*}
where we used the property \eqref{wshift}.
In general, for $B=B(n,n,\dots,n)=[n]\times [n]$, we have
\begin{equation}
r_n(a,b;q,p;B)=[n]_{aq^{-2n},bq^{-n};q,p}[n-1]_{aq^{2-2n},bq^{1-n};q,p}\cdots
[1]_{aq^{-2},bq^{-1};q,p}.
\end{equation}
See Corollary~\ref{nfac} for a proof.
\end{example}

The following lemma plays an essential role in the subsequent developments
leading to the product formula in Theorem~\ref{thm:elptprod}.
\begin{lemma}\label{lem:prod}
Let $B=B(b_1,\dots,b_n)\subset[n]\times\mathbb N$ be a board of $n$ columns
and let $B_k$ denote the extended board by attaching an $[n]\times [k]$
board below $B$ (the additional rows being indexed by $0,-1,\dots,{-k+1}$).
Suppose that $Q\in\mathcal{N}_t(B_k)$ is a rook placement of $t$ rooks
in the first $i-1$ columns of $B_k$. Let $D_i(Q)$ denote the set of all
rook placements which extend $Q$ by adding a rook in column $i$. Then we have 
\begin{equation}
\sum_{P\in D_i(Q)}wt(P)=
[b_i+k-t]_{aq^{2(i-1-b_i)},bq^{i-1-b_i};q,p}\,wt(Q).\label{eqn:lemma}
\end{equation}
\end{lemma}
\begin{proof} 
Let $i=1$. We want to show that
\begin{equation*}
\sum_{P\in D_1 (Q)}wt(P)=[b_1 +k]_{aq^{-2b_1},bq^{-b_1};q,p}.
\end{equation*}
If we consider all possible rook placements $P$ in the first column
and sum up all the weights of $P$, then we obtain 
\begin{align*}
\sum_{P\in D_1 (Q)}wt(P) &= 1+w_{aq^{-2b_1},bq^{-b_1};q,p}(1)+\cdots +
\prod_{j=1}^{b_1 +k-1}w_{aq^{-2b_1},bq^{-b_1};q,p}(j)\\
&=1+\sum_{j=1}^{b_1 +k-1}W_{aq^{-2b_1},bq^{-b_1};q,p}(j)\\
&=[b_1 +k]_{aq^{-2b_1},bq^{-b_1};q,p},
\end{align*}
where the sum telescopes according to \eqref{recelln}.

Now given $Q$, a rook placement of $t$ rooks in the first $i-1$ columns,
we consider all possible rook placements of one additional rook
in the $i$-th column.
If we place the $i$-th rook in the topmost possible place, then it cancels
all the empty cells below and so the weight coming from that rook placement
is $1$. Say we placed the $i$-th rook in the second topmost possible place.
Then there is one empty cell which was the topmost possible cell to place
a rook. If the coordinate of that cell is $(i, b_i-l_1)$, then that
means there are $l_1$ many rooks in the north-west region of that cell.
So the weight of this cell would be
\begin{align*}
w_{a,b;q,p}(i-b_i+l_1-l_1)&= w_{a,b;q,p}(i-1-b_i+1)\\&=
w_{aq^{2(i-1-b_i)},bq^{i-1-b_i};q,p}(1),
\end{align*}
by \eqref{wshift}. If we place the $i$-th rook in the
third topmost place, then the weight of the second empty cell would be
\begin{align*}
w_{a,b;q,p}(i-b_i+l_1+l_2+1-l_1-l_2)&=
w_{a,b;q,p}(i-1-b_i+2)\\&=w_{aq^{2(i-1-b_i)},bq^{i-1-b_i};q,p}(2),
\end{align*}
where $l_2$ is the number of rows between the topmost empty cell and the
second topmost empty cell. If we place the $i$-th rook in the bottom-most
possible cell, the weight coming from that placement would be
\begin{equation*}
\prod_{j=1}^{b_i+k-t-1}w_{aq^{2(i+n-b_i-1)},bq^{i+n-b_i-1};q,p}(j).
\end{equation*}
Hence by summing up all the weights coming from the all possible
rook placements of the $(t+1)$-st rook in the $i$-th column, we get 
\begin{align*}
1+\sum_{s=1}^{b_i+k-t-1}\prod_{j=1}^s w_{aq^{2(i-1-b_i)},bq^{i-1-b_i};q,p}(j)
&=1+\sum_{s=1}^{b_i+k-t-1} W_{aq^{2(i-1-b_i)},bq^{i-1-b_i};q,p}(s)\\
&=[b_i+k-t]_{aq^{2(i-1-b_i)},bq^{i-1-b_i};q,p}.
\end{align*}
Combining this with the weights coming from the placement $Q$, we obtain
\eqref{eqn:lemma}.
\end{proof}

The following proposition constitutes an elliptic extension
of Proposition~\ref{prop:q1}. As before, $B^\infty$ denotes the Ferrers
board obtained by appending below $B$ the infinite board of width $n$,
and for a rook placement $P$ in $B^\infty$, $\text{max}(P)$ denotes
the number of rows below the ground in which the lowest rook is located.
\begin{proposition}\label{prop:elp1}
For a Ferrers board $B=B(b_1,b_2,\dots, b_n)$, we have 
\begin{equation}\label{eqn:prop1}
\frac{1}{1-z}\sum_{P\in \mathcal{N}_n (B^{\infty})}z^{\text{max}(P)}\cdot wt(P)=
\sum_{k\ge 0}z^k \prod_{i=1}^n [k+b_i -i+1]_{aq^{2(i-1-b_i)},bq^{i-1-b_i};q,p}.
\end{equation}
\end{proposition}

\begin{proof}
We first show the following identity
\begin{equation}
\sum_{P\in \mathcal{N}_n(B^\infty)} wt(P)\cdot \chi(\text{max}(P)\le k) =
\prod_{i=1}^n [k+b_i -i+1]_{aq^{2(i-1-b_i)},bq^{i-1-b_i};q,p},\label{eqn:max}
\end{equation}
where $\chi$ is the truth function, i.e.\ $\chi(A)=1$ if the statement
$A$ is true, otherwise $\chi(A)=0$.
This easily follows from Lemma~\ref{lem:prod} by iteration,
using the fact that each column of $B^\infty$ contains a rook.
Then \eqref{eqn:prop1} is obtained by multiplying both sides
of \eqref{eqn:max} by $z^k$
and summing over all $k\ge 0$.
\end{proof}

The following product formula is the main result of this section. 
\begin{theorem}\label{thm:elptprod}
Let $B=B(b_1,\dots, b_n)$ be a Ferrers board. Then we have 
\begin{align}\label{eqn:prodthm}
\qquad\sum_{k=0}^n r_{n-k}(a,b;q,p;B)\,
\prod_{j=1}^k[z-j+1]_{aq^{2(j-1)},bq^{j-1};q,p}&\notag\\
=\prod_{i=1}^n [z+b_i -i+1]_{aq^{2(i-1-b_i)},bq^{i-1-b_i};q,p}&.
\end{align}
\end{theorem}

\begin{proof}
It suffices to prove the theorem for nonnegative integer values of $z$.
The result follows then by analytic continuation.

We consider the extended board $B_z$ by attaching an $[n]\times [z]$ board
below $B$ and compute 
\begin{equation}
\sum_{P\in \mathcal{N}_n(B_z)}wt(P)\label{eqn:wtsum}
\end{equation}
in two different ways.
On one hand, \eqref{eqn:wtsum} can be evaluated using
the $k=z$ case of \eqref{eqn:max} which thus explains
the right-hand side of \eqref{eqn:prodthm}.
On the other hand,
in \eqref{eqn:wtsum} we can consider, for each $0\le k\le n$,
the contributions from the $k$-rook
configurations below the ground, yielding
$\prod_{j=1}^k[z-j+1]_{aq^{2(j-1)},bq^{j-1};q,p}$,
and those from the $(n-k)$ rooks in $B$, yielding $r_{n-k}(a,b;q,p;B)$,
separately. This explains the
left-hand side of \eqref{eqn:prodthm}.
\end{proof}

The following corollary is an easy consequence of Theorem~\ref{thm:elptprod}.
\begin{corollary}\label{nfac}
Let $B=B(b_1,\dots, b_n)$ be a Ferrers board. Then we have 
\begin{equation}
r_n(a,b;q,p;B)
=\prod_{i=1}^n [b_i -i+1]_{aq^{2(i-1-b_i)},bq^{i-1-b_i};q,p}.
\end{equation}
In particular, for the square shape Ferrers board $B=B(n,n,\dots,n)=[n]\times[n]$,
we have
\begin{equation}
r_n(a,b;q,p;B)
=[n]_{aq^{-2n},bq^{-n};q,p}[n-1]_{aq^{2-2n},bq^{1-n};q,p}\dots[1]_{aq^{-2},bq^{-1};q,p}.
\end{equation}
\end{corollary}
\begin{proof}
In Theorem~\ref{thm:elptprod} we let $z\to 0$.
Since
\begin{equation*}
\prod_{j=1}^k[1-j]_{aq^{2(j-1)},bq^{j-1};q,p}=\delta_{k,0},
\end{equation*}
the left-hand side of \eqref{eqn:prodthm}
reduces to one term only, corresponding to $k=0$.
\end{proof}

We also establish
an elliptic analogue of Proposition~\ref{thm:recur},
a recursion for elliptic rook numbers.
\begin{theorem}\label{thm:elptrecur}
Let $B$ be a Ferrers board with $l$ columns of height at most $m$,
and $B\cup m$ denote the board obtained by adding the $(l+1)$-st
column of height $m$ to $B$. Then, for any integer $k$, we have 
\begin{subequations}
\begin{align}
r_k (a,b;q,p;B)={}&0\qquad\text{for $k<0$ or $k>l$,}\\
r_0 (a,b;q,p;B)={}&1\qquad\text{for $l=0$, i.e.\ for
$B$ being the empty board},\\\intertext{and}
r_k (a,b;q,p;B\cup m)=
{}&W_{aq^{2(l-m)},bq^{l-m};q,p}(m-k)\,r_{k}(a, b;q,p;B)\notag\\\label{recrk}
&+[m-k+1]_{aq^{2(l-m)},bq^{l-m};q,p}\,r_{k-1}(a, b;q,p;B).
\end{align}
\end{subequations}
\end{theorem}
\begin{proof}
This recursion stems from a weighted enumeration of placements of
$k$ nonattacking rooks on $B\cup m$. 
We distinguish the cases whether there is a rook
in the last column or not.
The first term on the right-hand side of \eqref{recrk} is obtained
when there is no rook in the last column.
The 
weight multiplied in front of $r_k(a,b;q,p;B)$ comes from the uncancelled $(m-k)$ cells
in the last column. The second term on the right-hand side of
\eqref{recrk} is obtained when there
is a rook in the last column. The coefficient in front of
$r_{k-1}(a,b;q,p;B)$ is a consequence of Lemma~\ref{lem:prod}.
\end{proof}

For $p\to 0$, followed by $b\to 0$ the above recurrence relation
\eqref{recrk} reads
\begin{equation}\label{recrkaq}
r_k(a;q;B\cup m)=W_{aq^{2(l-m)};q}(m-k)\,r_{k}(a;q;B)+
[m-k+1]_{aq^{2(l-m)};q}\,r_{k-1}(a;q;B),
\end{equation}
where according to \eqref{def:aqwt} and the $p\to 0$, then $b\to 0$
case of \eqref{elln},
\begin{equation}
W_{a;q}(k)=\frac{(1-aq^{1+2k})}{(1-aq)}q^{-k},\qquad\text{and}\qquad
[z]_{a;q}=\frac{(1-q^z)(1-aq^z)}{(1-q)(1-aq)}q^{1-z},
\end{equation}
and the $a;q$-rook numbers are given by
$$r_k(a;q;B)=\lim_{b\to 0}\left(\lim_{p\to 0}r_k(a,b;q,p;B)\right),$$
or
$$r_k(a;q;B)=\sum_{P\in \mathcal N _k(B)}\left(\prod_{(i,j)\in U_B (P)}w_{a;q}(i-j-r_{(i,j)}(P)) \right).$$
As an immediate consequence of this recursion,
we have the following product formula
for the $a;q$-rook numbers of a rectangular shape board $B=[l]\times[m]$
with $l$ columns and $m$ rows.

\begin{proposition}\label{prop:aqrect}
\begin{equation}
r_k(a;q;[l]\times[m])=
q^{\binom{k+1}{2}-lm}\begin{bmatrix}l\\k\end{bmatrix}_{q}
\frac{[m]_q !}{[m-k]_q !}
\frac{(aq^{l-m-k} ;q)_{k}(aq^{1+2l-2m};q^2)_{m-k}}{(aq^{1-2m};q^2)_{m}}.
\end{equation}
\end{proposition}
\begin{proof}
This follows by induction on $l$, the $l=0$ case being trivial.
In the computation of $r_k(a;q;[l+1]\times[m])$
as a sum of two explicit terms according to
the recurrence relation \eqref{recrkaq}, 
after pulling out common factors,
the sum of the two terms nicely factorizes due to the simple identity
\begin{align*}
(1-q^{l-k+1})(1-aq^{l-m-k})q^k+(1-q^k)(1-aq^{2l-m-k+1})&=(1-q^{l+1})(1-aq^{l-m}).\\[-3em]
\end{align*}
\end{proof}

\smallskip
Note that the elliptic rook numbers and even the $a,b;q$-rook numbers
(obtained from the elliptic rook numbers by letting $p\to 0$),
nor the $0,b;q$-rook numbers,
of rectangular shape boards in general do not factorize (unless $k=0$ or $k=l$).
They in fact already don't factorize in the case $l=2$ and $k=1$ (and $m>1$).

We now take a close look at several special cases of
elliptic rook numbers of particular interest.

\subsection{Elliptic Stirling numbers of the second kind}
The Stirling numbers of the second kind
admit a nice rook theoretic interpretation
when $B$ is a staircase board $\mathsf{St}_n=B(0,1,\dots, n-1)$
(see \cite[Corollary~2.4.2]{Stan}). Namely, for each configuration of
$n-k$ nonattacking rooks on $\mathsf{St}_n$, 
we can associate a set partition of $[n]$ in $k$ blocks. Whenever a cell $(i,j)$
is occupied by a rook, $i$ and $j$ are put in the same block, and 
the numbers which are not contained in any block in this way 
correspond to single blocks.
This describes a one-to-one correspondence between configurations of
$n-k$ nonattacking rooks on
$\mathsf{St}_n$ and  set partitions of $[n]$ into $k$ blocks.
Garsia and Remmel~\cite{GR} extended
this to the $q$-case, thus providing a rook theoretic realization
of Carlitz'~\cite{C1,C2} $q$-Stirling numbers.

We consider the staircase board $\mathsf{St}_n$ to define an
elliptic analogue of
the Stirling numbers of the second kind. For $b_i=i-1$, $i=1,\dots, n$,
Equation~\eqref{eqn:prodthm} becomes 
\begin{equation}
\left( [z]_{a,b;q,p}\right)^n=
\sum_{k=0}^n r_{n-k}(a,b;q,p;\mathsf{St}_n)\,\prod_{j=1}^k
[z-j+1]_{aq^{2(j-1)},bq^{j-1};q,p}.\label{eqn:elptStirling}
\end{equation}
The $r_{n-k}(a,b;q,p;\mathsf{St}_n)$ are actually the
\emph{elliptic Stirling numbers
of the second kind} $\mathcal{S}_{a,b;q,p}(n,k)$ which have recently been
defined and studied (in a different setting) by
Zs\'ofia Keresk\'{e}nyin\'{e} Balogh
and the first author~\cite{KBSchl}.

By using the $y=k$ case of the elementary identity \eqref{recellny},
we obtain from  \eqref{eqn:elptStirling} the following recursion
\begin{align}
\mathcal{S}_{a,b;q,p}(n,k)&=0\qquad\text{for $k<0$ or $k>n$},\notag\\
\mathcal{S}_{a,b;q,p}(0,0)&=1,\notag\\\intertext{and, for $k\ge 0$,}
\mathcal S_{a,b;q,p}(n+1,k)&= W_{a,b;q,p}(k-1)\mathcal S_{a,b;q,p}(n,k-1)+
[k]_{a,b;q,p}\mathcal S_{a,b;q,p}(n,k),
\end{align}
which also can be derived from Theorem~\ref{thm:elptrecur}. 

An explicit formula for the elliptic Stirling numbers
$\mathcal{S}_{a,b;q,p}(n,k)$ has not yet been established.
However, in \cite{KBSchl} the following formulae for small $k$
have been worked out.
\begin{subequations}\label{snk}
\begin{align}
\mathcal{S}_{a,b;q,p}(n,0)&=\delta_{n,0},\\
\mathcal{S}_{a,b;q,p}(n,1)&=1-\delta_{n,0},\\
\mathcal{S}_{a,b;q,p}(n,2)&=[2]_{a,b;q,p}^{n-1}-1,\\
\mathcal{S}_{a,b;q,p}(n,3)&=\frac 1{[2]_{aq^2,bq;q,p}}
\left([3]_{a,b;q,p}^{n-1}-[2]_{aq^2,bq;q,p}[2]_{a,b;q,p}^{n-1}+w_{a,b;q,p}(2)\right).
\end{align}
\end{subequations}
For $p\to0$, followed by $a\to 0$ and $b\to 0$,
these explicit evaluations can be easily seen to match
the special instances $k=0,1,2,3$ of Carlitz'~\cite[Equation~(3.3)]{C2}
well-known formula
\begin{equation}\label{carlitzexpl}
\mathcal S_q(n,k)=\frac 1{[k]_q!}\sum_{j=0}^k(-1)^jq^{\binom j2}
\begin{bmatrix}k\\j\end{bmatrix}_q[k-j]_q^n.
\end{equation}
As a matter of fact, the right-hand side of \eqref{carlitzexpl}
can also be rewritten in terms of basic hypergeometric series
(see \cite{GRhyp} for definitions and notation).
As such, the $q$-Stirling number of the second kind can be expressed
as the following multiple of a basic hypergeometric series of
Karlsson--Minton type:
\begin{equation}\label{stirlingkmtype}
\mathcal S_q(n,k)=\frac{[k]_q^ n}{[k]_q!}\,
_n\phi_{n-1}\!\left[\begin{matrix}
q^{1-k},q^{1-k},\dots,q^{1-k}\\q^{-k},\dots,q^{-k}\end{matrix};q,q^{k-n}\right].
\end{equation}
The existence of the latter series representation
is not so surprising, if one recalls that a big class of
($q$-)rook numbers generally admit a representation in terms
of (basic) hypergeometric series of Karlsson--Minton type, as revealed by
Haglund~\cite{H0}.

Coming back to our quest for finding an explicit formula in the elliptic case,
it is at this moment still not entirely clear how the pattern
in \eqref{snk} for the elliptic Stirling numbers can be extended to a
formula for $\mathcal S_{a,b;q,p}(n,k)$ valid for general $k$.

\subsection{Elliptic $r$-restricted Stirling numbers of the second kind}
\label{subsec:Str2}
The $r$-restricted Stirling numbers of the second kind count the number of
partitions of $[n]$ into $k$ blocks such that each of the first
$r$ numbers $1,2\dots,r$ is in a different block (cf.\ \cite{B}
or \cite{MS}).
The case $r=1$ (or $r=0$) gives the usual Stirling numbers of the second kind.
In the literature, the $r$-restricted  Stirling numbers of the second kind
are usually just called $r$-Stirling numbers of the second kind.
Nevertheless, in \cite[see sequences A143494, A143495 and A143496]{Sl}
they are referred to as ``$r$-restricted'', a terminology which we adopt here,
mainly to avoid confusion
with the $q$-Stirling numbers of the second kind.
These numbers admit a rook theoretic interpretation 
when $B$ is a cut-off staircase board
$\mathsf{St}^{(r)}_n=B(0,\dots,0,r,r+1,\dots,n-1)$
of $n$ columns, the first $r$ columns being empty.
The correspondence between the $n-k$ nonattacking
rook placements in $\mathsf{St}_n^{(r)}$ 
and the set partitions of $[n]$ in $k$ blocks works exactly in the same way
as for the board $\mathsf{St}_n$. Then the shape of the board $\mathsf{St}_n^{(r)}$ 
puts $1,2,\dots, r$ automatically in different blocks. 

We use $\mathsf{St}^{(r)}_n$ in \eqref{eqn:prodthm} to define an elliptic extension of 
the $r$-restricted Stirling numbers of the second kind.
For $b_i=0$ for $i=1,\dots,r$, and $b_i=i-1$ for $i=r+1,\dots, n$,
Equation~\eqref{eqn:prodthm} becomes
\begin{align}
&\left([z]_{a,b;q,p}\right)^{n-r}
\prod_{i=1}^r[z-i+1]_{aq^{2(i-1)},bq^{i-1};q,p}\notag\\&=
\sum_{k=0}^n r_{n-k}(a,b;q,p;\mathsf{St}^{(r)}_n)\,\prod_{j=1}^k
[z-j+1]_{aq^{2(j-1)},bq^{j-1};q,p}.\label{eqn:elptrStirling}
\end{align}
Defining $\mathcal{S}^{(r)}_{a,b;q,p}(n,k):=r_{n-k}(a,b;q,p;\mathsf{St}^{(r)}_n)$
to be the \emph{elliptic $r$-restricted Stirling numbers
of the second kind}, we obtain from Theorem~\ref{thm:elptrecur}
the following recursion
\begin{align}
\mathcal{S}^{(r)}_{a,b;q,p}(n,k)&=0\qquad\text{for $k<r-1$ or $k>n$},\notag\\
\mathcal{S}^{(r)}_{a,b;q,p}(r-1,r-1)&=1\qquad\text{(an artificial but
felicitous initial condition)},\notag\\\intertext{and,
for $k\ge r-1$,}
\mathcal S^{(r)}_{a,b;q,p}(n+1,k)&= W_{a,b;q,p}(k-1)\mathcal S_{a,b;q,p}^{(r)}(n,k-1)+
[k]_{a,b;q,p}\mathcal S_{a,b;q,p}^{(r)}(n,k).
\end{align}

\subsection{Elliptic Lah numbers}
The $q$-Lah numbers $\mathcal L_{n,k}(q)$ have first been studied
by Garsia and Remmel in \cite{GR0} by carrying out a $q$-counting of
placements of $n$ distinguishable balls in $k$ nonempty
indistinguishable tubes which have a linear order on its elements.
The same authors, in \cite{GR}, subsequently gave a rook theoretic
interpretation by considering the board
$\mathsf L_n=[n]\times[n-1]$ of $n$ columns, each of height $n-1$.
In this case, Proposition~\ref{thm:qrookthm} gives 
$$[z]_q\!\uparrow_n =\sum_{k=0}^n r_{n-k}(q;B)[z]_q\!\downarrow_k,$$
where $[z]_q\!\uparrow_n=[z]_q [z+1]_q\cdots [z+n-1]_q$ and
$[z]_q\!\downarrow_k=[z]_q[z-1]_q\cdots [z-k+1]_q$. If we let
$\mathcal{L}_{n,k}(q)=r_{n-k}(q;\mathsf L_n)$,
then the $q$-Lah numbers $\mathcal{L}_{n,k}(q)$ satisfy the recursion
\begin{equation}\label{reclah}
\mathcal{L}_{n+1,k}(q)=q^{n+k-1}\mathcal{L}_{n,k-1}(q)+
[n+k]_q \mathcal{L}_{n,k}(q).
\end{equation}
This can be established by placing $n+1-k$ nonattacking
rooks on $\mathsf L_{n+1}$ and 
distinguishing the cases whether there is a rook
or not in the union of the top row and the last column. 
If there is no such rook, we remove the top
row and last column (which contributes weight $q^{n+k-1}$) and
consider $n+1-k$ nonattacking rooks on the smaller board $\mathsf L_{n}$.
If, for $1\le j\le n+1$, there is a rook in the
$j$-th position of the
top row, the weight of the uncancelled cells coming from this rook
(which are located to the left of the rook) will be $q^{j-1}$.
We then remove the top row and $j$-th column and are left with
a smaller board on which $n-k$
nonattacking rooks are placed. If there is no rook
in the top row, there must be one in the last column (but not
on the most top of that column).
The weight of the uncancelled cells (the top row included)
coming from this rook
will be $q^{n+l-1}$, for some $2\le l\le k$, depending on the position
of the other $n-k$ rooks. The precise analysis is similar to that of
the proof of Lemma~\ref{lem:prod}.
After removing the top row and last column
(the possible weights adding up to $[n+k]_q$),
we are again left with a smaller board on which $n-k$
nonattacking rooks are placed.

Using the recursion in \eqref{reclah}, one can verify that
$\mathcal{L}_{n,k}(q)$ has the following closed form
\begin{equation}
\mathcal{L}_{n,k}(q)=q^{k(k-1)}\begin{bmatrix}n\\ k\end{bmatrix}_q
\frac{[n-1]_q !}{[k-1]_q !}.
\end{equation}

Now we turn to the elliptic setting.
For the board $\mathsf L_n$, Theorem~\ref{thm:elptprod} gives 
\begin{align}\label{eqn:elprook_rah1}
&[z+n-1]_{aq^{2-2n},bq^{1-n};q,p}
[z+n-2]_{aq^{4-2n},bq^{2-n};q,p}\cdots[z]_{a,b;q,p}\notag\\
&=\sum_{k=1}^n r_{n-k}(a,b;q,p;\mathsf L_n)\;
[z]_{a,b;q,p}[z-1]_{aq^{2},bq^{1};q,p}
\cdots
[z-k+1]_{aq^{2k-2},bq^{k-1};q,p}.
\end{align}
Let $\mathcal{L}_{n,k}(a,b;q,p)$ denote $r_{n-k}(a,b;q,p;\mathsf L_n)$.
This defines an elliptic analogue of Lah numbers and matches those which
have been defined and studied (in a different setting)
by Keresk\'{e}nyin\'{e} Balogh and the first author~\cite{KBSchl}.
Then by distinguishing whether there is a rook in the union
of the top row and the last column
of the board $\mathsf L_{n+1}$ or not,
we obtain the following recursion for
$\mathcal{L}_{n,k}(a,b;q,p)$ :
\begin{align}\label{recelllah}
\mathcal{L}_{n+1,k}(a,b;q,p)=
W_{aq^{-2n},bq^{-n};q,p}(n+k-1)\,\mathcal{L}_{n, k-1}(a,b;q,p)&\notag\\
+[n+k]_{aq^{-2n},bq^{-n};q,p}\,\mathcal{L}_{n,k}(a,b;q,p)&.
\end{align}
Unfortunately, this elliptic analogue of Lah number does not have a
nice closed form, but if we let $p\to 0$, followed by $b\rightarrow 0$,
then it has the following closed form
\begin{equation}
\mathcal{L}_{n,k}(a;q)=q^{\binom k2-\binom n2-n(k-1)}
\begin{bmatrix}n\\k\end{bmatrix}_q
\frac{[n-1]_q!}{[k-1]_q !}
\frac{(aq^{k-n+1};q)_{n+k}}{(aq^{3-2n};q^2)_n(aq^2;q^2)_k},
\end{equation}
the formula being a consequence of the
$(l,m,k)\mapsto(n,n-1,n-k)$ case of Proposition~\ref{prop:aqrect}.
It is not difficult to verify that the $a;q$-Lah numbers
$\mathcal{L}_{n,k}(a;q)$ converge to the $q$-Lah numbers $\mathcal{L}_{n,k}(q)$
when $a\rightarrow \infty$.


\begin{remark}
Goldman, Joichi and White \cite{GJW} observed that if
the left-hand sides of the product formula \eqref{eqn:rookthm}
are equal for two different Ferrers boards $B_1$ and $B_2$,
then also the rook numbers for $B_1$ and $B_2$ must be the same.
In this case the two Ferrers boards $B_1$ and $B_2$ are called
\emph{rook equivalent}. By appealing to the $q$-analogue of the factorization
theorem stated in Proposition~\ref{thm:qrookthm},
Garsia and Remmel~\cite{GR}
observed that Goldman, Joichi and White's observation
readily extends to the $q$-case, i.e., two Ferrers boards that have
the same rook numbers must also have the same $q$-rook numbers.
For instance, the $q$-Lah number $\mathcal{L}_{n,k}(q)$ can also be
obtained as the $q$-rook number of the Ferrers board $B(0,2,4,\dots,2n-2)$.
Theorem \ref{thm:elptprod} guarantees that this further
extends to the elliptic case, namely, two rook equivalent Ferrers boards have
the same elliptic rook numbers. In particular, 
\begin{equation*}
\mathcal{L}_{n,k}(a,b;q,p)=r_{n-k}(a,b;q,p;B_1)
=r_{n-k}(a,b;q,p;B_2),
\end{equation*}
for $B_1=\mathsf L_n=(n-1,\dots,n-1)$ ($n$ occurrences of $n-1$) and
$B_2=B(0,2,4,\dots,2n-2)$. This appears to be not at all obvious
from the combinatorial interpretation.
\end{remark}

\subsection{Elliptic $r$-restricted Lah numbers}

The $r$-restricted Lah numbers count the number of
placements of the elements $1,2,\dots,n$ into $k$ nonempty tubes
of linearly ordered elements such that $1,2,\dots,r$ are in distinct tubes,
cf.\ \cite{NR} or \cite{MS}.
The case $r=1$ (or $r=0$) gives the usual unsigned Lah numbers.
In the literature, the $r$-restricted Lah numbers are usually just called
$r$-Lah numbers. We use ''restricted'', in accordance
with the terminology used in
\cite[see sequences A143497, A143498 and A143499]{Sl},
to avoid confusion with the $q$-Lah numbers.
These numbers admit a rook theoretic interpretation 
when $B$ is the board
$\mathsf L^{(r)}_n=[n+r-1]\times[n-r]$
of $n+r-1$ columns, each of height $n-r$. In the following,
we describe a simple correspondence between the
rook configurations $P$ of $n-k$ nonattacking rooks on the
board $B=[n+r-1]\times[n-r]$ and the set of placements $T$
of the elements  $1,2,\dots,n$ into $k$ nonempty tubes of linearly ordered
elements such that the first $r$ numbers $1,2,\dots,r$ are in distinct tubes:
given a rook configuration of $n-k$ nonattacking rooks
on $\mathsf L^{(r)}_n$, we have $n-k$ rows containing rooks and
$k-r$ rows containing no rooks.
We start with the trivial tube placement $T_0=\{(1),(2),\dots,(r)\}$
of singletons and want to successively build up the
final tube placement by adding an element for each of the $n-r$ rows,
depending on the existence and the positions of the rooks.
Now, as mentioned, there are exactly $k-r$ rows without
rooks, say in rows $l_1,\dots,l_{k-r}$ (without
loss of generality, we may assume
$n-r\ge l_1>l_2>\dots>l_{k-r}\ge 1$).
These indices will determine the minimal elements of the
new tubes which we append to $T_0$. These $k-r$ additional
elements shall remain minimal elements, and we shall refer to them
as \emph{designated tube leaders}.
(On the contrary, the elements $1,2,\dots,r$ do not necessarily remain as
tube leaders in the final tube placement $T$.)
We thus replace $T_0$ by
$T_1=\{(1),(2),\dots,(r),(n+1-l_1),\dots,(n+1-l_{k-r})\}$,
and that is a new placement of exactly $k$ tubes
of singletons. The $n-k$ remaining rows in $P$ contain rooks
and are indexed by $[n-r]\setminus\{l_1,\dots,l_{k-r}\}$.
We add $r$ to each of these indices, thus obtain
the index set
$I=([n]\setminus[r])\setminus\{r+l_1,\dots,r+l_{k-r}\}$
which contain exactly the numbers of $[n]$ which have not
been already used in the tube $T_1$.
We remove the top-most rook in $P$, say $\mathbf r_1$,
and identity it with the smallest element in $I$, say $\iota_1$.
Since there are $n+r-1$ columns in $B$ of which $n-k-1$
columns contain rooks below $\mathbf r_1$, there are
exactly $k+r$ possibilities for $\mathbf r_1$ to be
placed in its row. On the other side, there are exactly
$(n+r-1)-(n-k-1)=k+r$ possible positions for the smallest
element $\iota_1$
of $I$ to be added to $T_1$. That is, $\iota_1$ can be placed
on top of any element (which gives $k$ possibilities),
or below any element except the $k-r$ designated
tube leaders (which gives $r$ additional possibilities).
In total we have $k+r$ possible positions to insert
$\iota_1$ in $T_1$, after which we obtain $T_2$.
We now remove $\mathbf r_1$ from $P$ and also delete
$\iota_1$ from $I$. In $P$, we turn to the next row from
the top containing a rook, say $\mathbf r_2$, remove it
and identify it with the next smallest element in
$I\setminus\{\iota_1\}$ which we label $\iota_2$.
Now there are $k+r+1$ possibilities
to place $\mathbf r_2$ in its row, and there are also
exactly $k+r+1$ possibilities for $\iota_2$ to be inserted
in $T_2$. We iterate this, and in the end have
$n+r-1$ possibilities for the $(n-k)$-th rook, say
$\mathbf r_{n-k}$ to be placed in its row, and accordingly,
$n+r-1$ possible positions to insert the maximal element
of $I$, say $\iota_{n-k}$, in the placement $T_{n-k}$
of tubes after which we finally obtain the final tube placement $T$.

In total we have
$$
\binom{n-r}{k-r}\frac{(n+r-1)!}{(k+r-1)!}=\binom{n+r-1}{k+r-1}\frac{(n-r)!}{(k-r)!}
$$
such placements. This number matches the $r$-restricted Lah number.

For a concrete example of a rook configuration mapped to a placement
of elements in tubes, see Figure~\ref{fig:rlah}, where we have chosen
$n=8$, $r=2$, $k=4$.
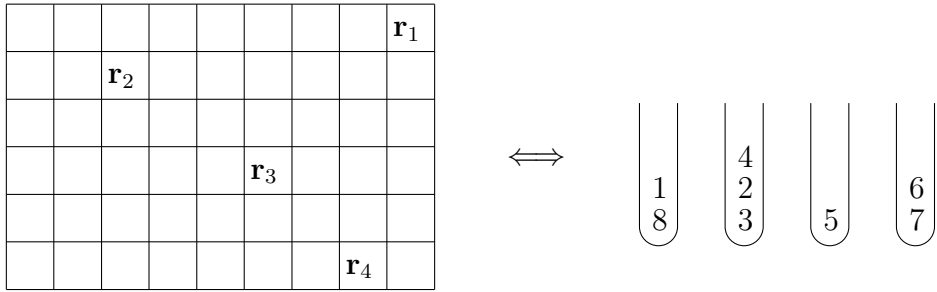
\begin{figure}[ht]
\begin{picture}(135,100)(15,15)
\multiput(15,15)(0,15){7}{\line(1,0){135}}
\multiput(15,15)(15,0){10}{\line(0,1){90}}
\put(137, 95){$\mathbf{r}_1$}
\put(47, 80){$\mathbf{r}_2$}
\put(92, 50){$\mathbf{r}_3$}
\put(122, 20){$\mathbf{r}_4$}
\end{picture}
\qquad\raisebox{4em}{$\Longleftrightarrow$}\qquad\quad
\raisebox{6em}{\begin{picture}(100,40)(10,10)
\put(10,10){\oval(12,90)[b]}
\put(37,10){\oval(12,90)[b]}
\put(64,10){\oval(12,90)[b]}
\put(91,10){\oval(12,90)[b]}
\put(8,-30){$8$}
\put(8,-20){$1$}
\put(35,-30){$3$}
\put(35,-20){$2$}
\put(35,-10){$4$}
\put(62,-30){$5$}
\put(89,-30){$7$}
\put(89,-20){$6$}
\end{picture}}
\caption{$n=8$, $r=2$, $k=4$; $B=[9]\times[6]$}\label{fig:rlah}
\end{figure}
We consider a placement of $4$ nonattacking rooks on $B=[9]\times[6]$,
the rooks being in the cells $(9,6)$, $(3,5)$, $(6,3)$, and $(8,1)$
(from top to bottom).
We start with putting the numbers $1$ and $2$ into the (the first)
two distinct tubes. Since the third and fifth row from the top of $B$
contain no rooks, we put the third and fifth smallest numbers not
already used, i.e.\ $5$ and $7$, into the third and fourth tubes,
respectively.
The numbers $5$ and $7$ are designated tube leaders. They will remain
to be the minimal elements of their tubes.
Now consider the top-most rook, $\mathbf r_1$, which is in $(9,6)$.
If we remove it and decide to put it back into the top row,
we have $6$ possibilities. The possible positions are
$(1,6)$, $(2,6)$, $(4,6)$, $(5,6)$, $(7,6)$ and $(9,6)$.
From these $(9,6)$ is the sixth position.
Accordingly, there are $6$ different positions for the smallest number
not already used in one of the tubes, i.e. $3$,
to be placed in one of the tubes.
The six different choices are as follows: the element $3$ can be put on
top of the element $1$, on top of the element $2$, on top of the
element $5$, on top of the element $7$, below the element $1$,
or below the element $2$. (Nothing can be put
below $5$ or below $7$ since they are designated tube leaders.)
We take the sixth option (as $(9,6)$ was the sixth possible position
of its row), which means that we put $3$ below $2$ (which is in the
second tube). We chop off the top row of the board and consider
the next rook from the top, $\mathbf r_2$, which is in position $(3,5)$.
This rook occupies the third possible position in its row, of $7$
possible positions in total.
Accordingly, we put the element $4$ (which is the smallest number
not already used in the tube placement) on top of the element $2$
which is the third possible position of $7$ possibilities (on top of
$1$, on top of $3$, on top of $2$, on top of $5$, on top of $7$, below $1$,
or below $2$.)
We again chop off the top row and also the empty row below it and
turn to the next rook from the top, $\mathbf r_3$, which is in
position $(6,3)$.
This rook occupies the sixth possible position in its row, of $8$
possible positions in total.
Accordingly, we put the element $6$ (which is the smallest number
not already used in the tube placement) on top of the element $7$
which is the sixth possible position of $8$ possibilities (on top of
$1$, on top of $3$, on top of $2$, on top of $4$, on top of $5$,
on top of $7$, below $1$, or below $3$.)
We again chop off the top row and also the empty row below it and turn
to the last remaining rook, $\mathbf r_4$, 
which is in position $(8,1)$. This is the eighth possible position
of $9$ possible positions in total. In the tube placement, the
eight possibility is the position below the element $1$.
Thus we put the last element not yet appearing in the placements of tubes,
i.e., the element $8$, below $1$. Finally we have arrived at the
placement $\{(8,1),(3,2,4),(5),(7,6)\}$, written as a set of ordered lists.
\begin{remark}
Note that in the case when $r=1$, this algorithm reduces to the case of the original
Lah numbers $\mathcal{L}_{n,k}$ which counts the number of ways placing 
$n$ distinguishable balls in $k$ nonempty tubes.
Garsia and Remmel \cite{GR0} also provided a correspondence between 
$r_{n-k}(\mathsf L_n)$ and a placement of $n$ balls in $k$ tubes, 
but their correspondence is different from ours explained above. 
\end{remark}

We use the board $\mathsf L_n ^{(r)}$ in Theorem \ref{thm:elptprod} to define 
an elliptic analogue of the $r$-restricted Lah numbers.
For $B=\mathsf L_n^{(r)}=[n+r-1]\times[n-r]$, i.e.\ $b_i=n-r$
for $i=1,\dots,n+r-1$,
Theorem~\ref{thm:elptprod} becomes
\begin{align}
\sum_{k=2r-1}^{n+r-1}r_{n+r-1-k}(a,b;q,p;\mathsf L_n^{(r)})
\prod_{j=1}^k[z-j+1]_{aq^{2(j-1)},bq^{j-1};q,p}&\notag\\=
\prod_{i=1}^{n+r-1}[z+n-r-i+1]_{aq^{2(i-1-n+r)},bq^{i-1-n+r};q,p}&
\end{align}
where we can readily start the index of summation with $k=2r-1$ since the lower
terms vanish,
and after shifting the index as $k\mapsto k+r-1$ and cancelling common
factors on both sides of the sum we obtain
\begin{align*}
\sum_{k=r}^{n}r_{n-k}(a,b;q,p;\mathsf L_n^{(r)})
\prod_{j=1}^{k}[z-r-j+2]_{aq^{2(j+r-2)},bq^{j+r-2};q,p}&\notag\\=
\prod_{i=1}^{n-r}[z+n-r-i+1]_{aq^{2(i-1-n+r)},bq^{i-1-n+r};q,p}
\prod_{i=1}^{r}[z-r-i+2]_{aq^{2(i+r-2)},bq^{i+r-2};q,p}.&
\end{align*}
For a more compact result, we replace $(a,b,z)$ by $(aq^{2(1-r)},bq^{1-r},z+r-1)$,
after which we obtain
\begin{align}
\sum_{k=r}^{n}r_{n-k}(aq^{2(1-r)},bq^{1-r};q,p;\mathsf L_n^{(r)})
\prod_{j=1}^{k}[z-j+1]_{aq^{2(j-1)},bq^{j-1};q,p}&\notag\\=
\prod_{i=1}^{n-r}[z+n-i]_{aq^{2(i-n)},bq^{i-n};q,p}
\prod_{i=1}^{r}[z-i+1]_{aq^{2(i-1)},bq^{i-1};q,p}.&
\end{align}
The \emph{elliptic $r$-restricted Lah numbers}
$\mathcal{L}^{(r)}_{n,k}(a,b;q,p):=r_{n-k}(aq^{2(1-r)},bq^{1-r};q,p;\mathsf L^{(r)}_n)$
satisfy the recursion
\begin{align}\label{recellrlah}
\mathcal{L}^{(r)}_{n+1,k}(a,b;q,p)=
W_{aq^{-2n},bq^{-n};q,p}(n+k-1)\,\mathcal{L}^{(r)}_{n, k-1}(a,b;q,p)&\notag\\
+[n+k]_{aq^{-2n},bq^{-n};q,p}\,\mathcal{L}^{(r)}_{n,k}(a,b;q,p)&,
\end{align}
with initial conditions
\begin{align}
\mathcal{L}^{(r)}_{n,k}(a,b;q,p)&=0\qquad\text{for $k<r-1$ or $k>n$},\notag\\
\mathcal{L}^{(r)}_{r-1,r-1}(a,b;q,p)&=1\qquad\text{(an artificial but
felicitous initial condition)}.
\end{align}

As in the elliptic Lah-number case, this elliptic analogue of $r$-restricted Lah number
does not have a nice closed form, but if we let $p\to 0$, followed by
$b\rightarrow 0$,
then it has the following closed form
\begin{equation}
\mathcal{L}^{(r)}_{n,k}(a;q)=
q^{\binom k2-\binom n2-n(k-1)+2\binom r2}
\begin{bmatrix}n+r-1\\k+r-1\end{bmatrix}_{q}
\frac{[n-r]_q !}{[k-r]_q !}
\frac{(aq^{1-n+k} ;q)_{n-k}(aq^{1+2r};q^2)_{k-r}}{(aq^{3-2n};q^2)_{n-r}},
\end{equation}
the formula being a consequence of the
$(a,l,m,k)\mapsto(aq^{2(1-r)},n+r-1,n-r,n-k)$ case of
Proposition~\ref{prop:aqrect}.
For $a\rightarrow \infty$ this $a;q$-analogue of
$r$-restricted Lah numbers $\mathcal{L}^{(r)}_{n,k}(a;q)$
converges to the following $q$-analogue of $r$-restricted Lah numbers
$\mathcal{L}^{(r)}_{n,k}(q)$ 
\begin{equation}
\mathcal{L}^{(r)}_{n,k}(q)=
q^{k(k-1)-r(r-1)}
\begin{bmatrix}n+r-1\\k+r-1\end{bmatrix}_{q}
\frac{[n-r]_q !}{[k-r]_q !}.
\end{equation}

\subsection{ $\mathfrak{p},q$-Analogues}\label{subsec:pq}
Briggs and Remmel~\cite{BR} defined the $\mathfrak{p},q$-analogue\footnote{In
this subsection, we use the Fraktur letter $\mathfrak{p}$ to denote
the second base variable instead
of the common Latin-script letter $p$ since in our elliptic setting,
we have reserved $p$ to denote the $nome$.}
of rook numbers by using the (homogeneous)
$\mathfrak{p},q$-analogue of $n$ and $n!$
defined by
$$[n]_{\mathfrak{p},q}:=\mathfrak{p}^{n-1}+\mathfrak{p}^{n-2}q +\cdots +
\mathfrak{p}q^{n-2}+q^{n-1}=\frac{\mathfrak{p}^n - q^n}{\mathfrak{p}-q}$$
and $[n]_{\mathfrak{p},q}!=[n]_{\mathfrak{p},q}[n-1]_{\mathfrak{p},q}\cdots
[1]_{\mathfrak{p},q}$. They in particular proved that for a
Ferrers board $B=B(b_1,\dots, b_n)\subseteq [n]\times \mathbb N$,
one has
\begin{subequations}
\begin{align}
\prod_{i=1}^n [z+b_i-(i-1)]_{\mathfrak{p},q}=\sum_{k=0}^n r_{k,n}(B,\mathfrak{p},q)\,p^{zk+\binom{k+1}{2}}
\prod_{i=0}^{n-k}[z-i+1]_{\mathfrak{p},q}
,
\end{align}
where
\begin{equation}
r_{k,n}(B,\mathfrak{p},q):=\sum_{P\in \mathcal{N}_k (B)}
q^{\alpha_B (P)+\epsilon_B (P)}\mathfrak{p}^{\beta_B (P)-(c_1+\cdots + c_k)},
\end{equation}
\end{subequations}
for specifically defined $\alpha_B (P)$, $\beta_B (P)$ and $\epsilon_B (P)$,
and where the $c_1,c_2,\dots,c_k$ are the column labels of the $k$ columns
containing rooks of $P$. See \cite{BR} for the full details.

If in \eqref{def:smallelpwt} we let $p=0$ and replace $q$ by
$q/\mathfrak{p}$, then the small weight function becomes
\begin{equation}\label{def:pqweight}
w_{a,b;q/\mathfrak{p},0}(k)=\frac{(\mathfrak{p}^{2k+1}-aq^{2k+1})
(\mathfrak{p}^k-bq^k)(b\mathfrak{p}^{k-2}-aq^{k-2})}
{(\mathfrak{p}^{2k-1}-aq^{2k-1})(\mathfrak{p}^{k+2}-bq^{k+2})
(b\mathfrak{p}^k-aq^k)}\mathfrak{p}q,
\end{equation}
while the $\mathfrak{p},q$-numbers become
\begin{equation}\label{ellnpq}
[z]_{a,b;q/\mathfrak{p},0}=\frac{(\mathfrak{p}^z-q^z)(\mathfrak{p}^z-aq^z)
(\mathfrak{p}^2-bq^2)(b-a)}
{(\mathfrak{p}-q)(\mathfrak{p}-aq)(\mathfrak{p}^{z+1}-bq^{z+1})
(b\mathfrak{p}^{z-1}-aq^{z-1})}.
\end{equation}

The use of this weight in Theorem \ref{thm:elptprod} yields
an $a,b$-extension of the above result
of Briggs and Remmel. By utilizing the weight function in
\eqref{def:pqweight} with the
staircase board $\mathsf{St}_n$, an $a,b$-extension of the
$\mathfrak{p},q$-Stirling numbers defined by Wachs and White~\cite{WaW}
can be obtained.

\section{$\J$-attacking rook model}\label{sec:j}

We develop an elliptic analogue of the \emph{$\J$-attacking rook model}
of Remmel and Wachs \cite{RW}. We recall their setting first.
For a fixed integer $\J\ge 1$, we say that a Ferrers board
$B(b_1,\dots, b_n)$ is a \emph{$\J$-attacking board} if for all
$1\le i<n$, $b_i\ne 0$ implies $b_{i+1}\ge b_i +\J-1$.
Suppose that $B(b_1,\dots, b_n)$ is a $\J$-attacking board and $P$ 
is a placement of rooks in $B(b_1,\dots, b_n)$ which has at most one
rook in each column of $B(b_1,\dots, b_n)$. Then for any individual
rook $\mathbf{r}\in P$, 
we say that $\mathbf{r}$ \emph{$\J$-attacks} a cell
$c\in B(b_1,\dots, b_n)$ if $c$ lies in a column which is
strictly to the right of the column of $\mathbf{r}$ and $c$
lies in the first $\J$ rows which are weakly above the row of
$\mathbf{r}$ and which are not $\J$-attacked by any rook which
lies in a column that is strictly to the left of 
$\mathbf{r}$. Figure~\ref{fig:jattack} shows an example of
$\J$-attack when $\J=2$. In Figure \ref{fig:jattack}, the cells which are attacked by the rook
$\mathbf{r}_i$ are denoted by $i$ in the cell. 
\begin{figure}[ht]
\begin{picture}(80,100)(10,10)
\multiput(10,10)(0,10){2}{\line(1,0){70}}
\multiput(20,30)(0,10){1}{\line(1,0){60}}
\multiput(30,40)(0,10){1}{\line(1,0){50}}
\multiput(40,50)(0,10){2}{\line(1,0){40}}
\multiput(50,70)(0,10){2}{\line(1,0){30}}
\multiput(60,90)(0,10){1}{\line(1,0){20}}
\multiput(70,100)(0,10){1}{\line(1,0){10}}
\multiput(10,10)(10,0){1}{\line(0,1){10}}
\multiput(20,10)(10,0){1}{\line(0,1){20}}
\multiput(30,10)(10,0){1}{\line(0,1){30}}
\multiput(40,10)(10,0){1}{\line(0,1){50}}
\multiput(50,10)(10,0){1}{\line(0,1){70}}
\multiput(60,10)(10,0){1}{\line(0,1){80}}
\multiput(70,10)(10,0){2}{\line(0,1){90}}
\put(22, 23){$\mathbf{r}_1$}
\put(32,22){$1$}
\put(42,22){$1$}
\put(52,22){$1$}
\put(62,22){$1$}
\put(72,22){$1$}
\put(32,32){$1$}
\put(42,32){$1$}
\put(52,32){$1$}
\put(62,32){$1$}
\put(72,32){$1$}
\put(41, 13){$\mathbf{r}_2$}
\put(52, 12){$2$}
\put(62, 12){$2$}
\put(72, 12){$2$}
\put(52, 42){$2$}
\put(62, 42){$2$}
\put(72, 42){$2$}
\put(61, 63){$\mathbf{r}_3$}
\put(72, 62){$3$}
\put(72, 72){$3$}
\end{picture}
\caption{$\J=2$, $B=B(1,2,3,5,7,8,9)$.}\label{fig:jattack}
\end{figure}
Let a rook $\mathbf{r}$ in $B(b_1,\dots, b_n)$ cancel the cells below
it and the cells which are $\J$-attacked by $\mathbf{r}$.
A placement $P$ of $k$ rooks in $B$ is called 
\emph{$\J$-nonattacking} if each column contains at most one rook and each rook
does not $\J$-attack other rooks.
Given a $\J$-attacking board $B$, we let $\mathcal{N}_k ^\J (B)$
be the set of all placements $P$ of $k$ $\J$-nonattacking rooks in $B$. 
  
Let $B=B(b_1,\dots, b_n)$ be a $\J$-attacking board. For any
placement $P\in\mathcal{N}_k ^\J (B)$, denote the number of uncancelled
cells in $B-P$ as $u_B ^{\J}(P)$. We define the $q$-rook number of $B$ by 
$$r_k ^{\J} (q;B)=\sum_{P\in \mathcal{N}_k ^\J(B)}q^{u_B ^{\J}(P)}.$$
Then Remmel and Wachs \cite{RW} proved the following product formula.  

\begin{theorem}\cite{RW}\label{thm:RW}
Let $B=B(b_1,\dots, b_n)$ be a $\J$-attacking board. Then 
$$\prod_{i=1}^n [z+b_i -\J(i-1)]_q =
\sum_{k=0}^n r_{n-k} ^\J (q;B)[z]_q \!\downarrow_{k,\J },$$
where $[z]_q\!\downarrow_{0,\J}=1$ and for $k>0$,
$[z]_q \!\downarrow_{k,\J}=[z]_q [z-\J]_q\cdots [z-(k-1)\J]_q$. 
\end{theorem}
\begin{remark}
Remmel and Wachs defined $(\mathfrak{p},q)$-rook numbers including one more
parameter $\mathfrak{p}$. Here we have set $\mathfrak p=1$. As we remarked in 
Section \ref{subsec:pq}, we can modify the weight function to recover 
or extend the $\mathfrak{p},q$-rook numbers.
\end{remark}

Now we establish an elliptic analogue of the $\J$-attacking rook model.
Given a $\J$-attacking board $B=B(b_1,\dots, b_n)$ and a placement
$P\in \mathcal{N}_k ^\J (B)$, let $U_B ^{\J}(P)$ be the set of uncancelled
cells in $B-P$. Then define 
$$wt^{\J}(P)=\prod_{(i,j)\in U_B ^{\J}(P)}w_{a,b;q,p}(\J (i-1)+1-j-\J r_{(i,j)}(P)),$$
where $r_{(i,j)}(P)$ is the number of rooks in $P$ which are in the
north-west region of $(i,j)$, and define the $k$-th rook number of $B$ by
$$r_{k}^{\J}(a,b;q,p;B)=\sum_{P\in\mathcal{N}_k ^\J (B) }wt^{\J}(P).$$
Then we have the following elliptic analogue of Theorem \ref{thm:RW}.

\begin{theorem}
Let $B=B(b_1,\dots, b_n)$ be a $\J$-attacking board. Then we have 
\begin{align}\label{eqn:elptjump}\notag
&\prod_{i=1}^n [z+b_i -\J(i-1)]_{aq^{2(\J (i-1) -b_i )},bq^{ \J (i-1) -b_i };q,p}\\
&=\sum_{k=0}^n r_{n-k}^{\J}(a,b;q,p;B)
\prod_{j=1}^k [z-\J (j-1)]_{aq^{2\J( j-1)},bq^{\J (j -1)};q,p}.\qquad\qquad\qquad
\end{align}
\end{theorem}

\begin{proof}
The idea of proof is basically the same as in the proof of the product
formula in the case when $\J=1$. It is enough to prove \eqref{eqn:elptjump}
for all positive integers $z\ge \J n$. So fix a positive integer $z\ge \J n$
and consider $B_z$, the extended board obtained from $B$ by attaching
a $[n]\times[z]$ board below $B$.
We shall consider nonattacking placements of $n$ rooks
in $B_z$. Recall that we denoted the line separating the board $B$ and the
extended part below by $\mathfrak{g}$ and called it \emph{ground}. 
A rook $\mathbf{r}$ placed in $B$ will $\J$-attack as described above,
and thus it only $\J$-attacks cells which are above the ground.
If a rook $\mathbf{r}$ is placed below the ground, then it shall
$\J$-attack only the cells below the ground. More precisely, if a rook
$\mathbf{r}$ is placed below the ground $\mathfrak{g}$, then it
$\J$-attacks cells strictly to the right of the column containing
$\mathbf{r}$ and in the first $\J$ rows which are weakly above the
row containing $\mathbf{r}$ and below the ground which contain no cells
that are $\J$-attacked by any other rook $\mathbf{r}'$ to the left of
$\mathbf{r}$ if there are such $\J$ rows, and if there are $t<\J$ such rows,
then the rook $\mathbf{r}$ $\J$-attacks those $t$ rows and the first $\J-t$
rows below the row of $\mathbf{r}$ which contain no cells that are
$\J$-attacked by any other rooks to the left of $\mathbf{r}$.
Then we define that a rook placed below the ground cancels the cells below it
and the cells which are $\J$-attacked by the rook. 

Now let $\mathcal{N}_n ^\J(B_z)$ denote the set of all placements $P$ of
$n$ rooks in $B_z$ such that there is at most one rook in each row and
column and no rooks $\J$-attack another rook. For a placement
$P\in \mathcal{N}_n ^\J(B_z)$, let $U_{B_z} ^{\J}(P)$ be the set of
uncancelled cells in $B_z-P$. Then we show \eqref{eqn:elptjump}
by computing the sum 
\begin{equation}\label{eqn:jumppf}
\sum_{P\in \mathcal{N}_n ^\J (B_z)}wt^{\J}(P),
\end{equation}
where 
\begin{equation}\label{eqn:jwt}
wt^{\J}(P)=
\prod_{(i,j)\in U_{B_z} ^{\J}(P)}w_{a,b;q,p}(\J (i-1)+1-j-\J r_{(i,j)}(P)),
\end{equation}
in two different ways. 
First, we place rooks column by column from the left to right and compute
the contribution to the sum. If we place a rook in the first column
in all possible ways, it gives the weights 
\begin{align*}
&1+w_{a,b;q,p}(1-b_1)+\cdots +w_{a,b;q,p}(1-b_1)\cdots w_{a,b;q,p}(z-1)\\
={}& 1+w_{aq^{-2b_1},bq^{-b_1}}(1)+ w_{aq^{-2b_1},bq^{ -b_1}}(1)\cdots
w_{aq^{-2b_1},bq^{-b_1}}(b_1+z-1)\\
={}& [b_1 +z]_{aq^{-2b_1},bq^{ -b_1};q,p}.
\end{align*}
This first rook cancels $\J$ rows to the right of it and so,
placing the second rook in the second column gives
$[b_2 +z-\J]_{aq^{2(\J-b_2)},bq^{\J -b_2}}$. Note that the weights sum up
to the elliptic integer since the factor $\J r_{(i,j)}(P)$ compensates
the possible gap in the row coordinates due to the cancellation from
the rook to the left. Placing $n$ rooks in this way gives the
left-hand side of \eqref{eqn:elptjump}. 

For the right hand side of
\eqref{eqn:elptjump}, we place $n-k$ rooks in $B$ and $k$ rooks in the
extended part below the ground and compute \eqref{eqn:jumppf}.
Fix a placement $\mathcal{Q}$ of $n-k$ rooks in $B$. We want to compute 
$$W(\mathcal{Q})=\sum_{\substack{P\in \mathcal{N}_n ^{\J}(B_z),\\
P\cap B=\mathcal{Q}}}
wt^\J (P).$$
We put $wt^\J (\mathcal{Q})$ for the weight contribution coming from
the uncancelled cells in $B$ and compute the weight coming from the
uncancelled cells below the ground. Let $s$ denote the first available
column coordinate for the first rook below the ground. This means that there
are $s-1$ rooks in the north-west region of this column.
Then the possible placements of the first rook in this column give
\begin{equation*}
1+w_{a,b;q,p}(\J (s-1)+1 -\J (s-1))+\cdots +
w_{a,b;q,p}(1)\cdots w_{a,b;q,p}(z-1)
= [z]_{a,b;q,p}.
\end{equation*}
This rook cancels $\J$ rows to the right and so the second rook
contributes $[z-\J]_{aq^{2\J}, bq^{\J };q,p}$. Finally, placing $k$ rooks
below the ground gives 
$$\prod_{j=1}^k [z-\J (j-1)]_{aq^{2\J (j-1)},bq^{\J( j -1)};q,p}$$
to $wt^\J (P)$. Hence, 
\begin{align*}
\sum_{\mathcal{Q}\in \mathcal{N}_{n-k}^\J (B)}W(\mathcal{Q}) &=
\sum_{\mathcal{Q}\in \mathcal{N}_{n-k}^\J (B)}
\sum_{\substack{P\in \mathcal{N}_n ^{\J}(B_z),\\P\cap B=\mathcal{Q}}}wt^\J (P)\\
&= \sum_{\mathcal{Q}\in \mathcal{N}_{n-k}^\J (B)} wt^\J (\mathcal{Q})
\prod_{j=1}^k [z-\J (j-1)]_{aq^{2\J (j-1)},bq^{\J (j -1)};q,p}\\
&= r_{n-k} ^\J (a,b;q,p;B) \prod_{j=1}^k [z-\J (j-1)]_{aq^{2\J( j-1)},bq^{\J (j -1)};q,p}.
\end{align*}
We get the right hand side of \eqref{eqn:elptjump} by
summing this over $k=0,\dots, n$. 
\end{proof}

It is clear that by taking $z=0$ in \eqref{eqn:elptjump} the following  product formula
is obtained.
\begin{corollary}
Let $B=B(b_1,\dots, b_n)$ be a $\J$-attacking board. Then we have 
\begin{equation*}
 r_{n}^{\J}(a,b;q,p;B)=
\prod_{i=1}^n [b_i -\J(i-1)]_{aq^{2(\J (i-1) -b_i )},bq^{ \J (i-1) -b_i };q,p}.
\end{equation*}
\end{corollary}

\subsection{Elliptic analogue of generalized Stirling numbers
of the second kind}\label{subsec:pqstir}

Here we consider the generalized $(\mathfrak p,q)$-Stirling numbers of the
second kind $\tilde{S}_{n,k}^{\I,\J}(\mathfrak p, q)$ 
(here $\I$ is an additional nonnegative integer parameter) which were thoroughly
investigated  by Remmel and Wachs~\cite{RW}. They are defined by 
\begin{equation}
\tilde{S}_{n+1,k}^{\I,\J}(\mathfrak p, q)=
q^{\I+(k-1)\J}\tilde{S}_{n, k-1}^{\I, \J}(\mathfrak p, q)+
\mathfrak p^{-(n+1)\J}[k\J+\I]_{\mathfrak p, q} \tilde{S}_{n,k}^{\I,\J}(\mathfrak p, q),
\end{equation}
with $\tilde{S}_{0,0}^{\I, \J}(\mathfrak p, q)=1$ and $\tilde{S}_{n,k}^{\I, \J}(\mathfrak p, q)=0$
if $k<0$ or $k>n$. They also satisfy 
$$
[z+\I]_{\mathfrak p,q} ^n =\sum_{k=0}^n \tilde{S}_{n,k}^{\I, \J}(\mathfrak p, q)
\mathfrak p^{z(n-k)+\binom{n-k+1}{2}}[z]_{\mathfrak p,q}\!\downarrow_{k,\J}.
$$
While for $\I=0$ and $\J=1$
they were introduced by Wachs and White \cite{WaW},
the rescaled variant $S_{n,k}^{\I, \J}(\mathfrak p, q)=
\mathfrak p^{\binom{n-k+1}{2}\J-(n-k)(\I-1)}
q^{-k\I-\binom{k}{2}\J}\tilde{S}_{n,k}^{\I, \J}(\mathfrak p, q)$
was defined by de~M\'{e}dicis and Leroux~\cite{ML} by using $0$-$1$ tableaux. 

From now on, we set $\mathfrak p=1$ and use the notation
$\tilde{S}_{n,k}^{\I, \J}(q)=\tilde{S}_{n,k}^{\I, \J}(1,q)$ and
$S_{n,k}^{\I, \J}(q)=S_{n,k}^{\I, \J}(1,q)$ for
$\tilde{S}_{n,k}^{\I, \J}(q)=q^{k\I+\binom{k}{2}\J}S_{n,k}^{\I, \J}(q)$. 

Let $B_{\I, \J, n}=B(\I, \I+\J, \I+2\J, \dots, \I+(n-1)\J)$. 
In \cite{RW}, Remmel and Wachs showed that 
$$\tilde{S}_{n,k}^{\I, \J}(q)= r_{n-k}^{\J}(q;B_{\I, \J, n}).$$
If we let $\alpha_B (P)$ denote the number of uncancelled cells of $B$
which lie above a rook in $P$, then Remmel and Wachs also showed that 
$$S_{n,k}^{\I, \J}(q)=\sum_{P\in\mathcal{N}_{n-k}^\J (B_{\I, \J, n})}q^{\alpha_B (P)}.$$

We use the board $B=B_{\I, \J, n}$ in \eqref{eqn:elptjump} to define
an elliptic analogue of $\tilde{S}_{n,k}^{\I, \J}(q)$. For this board,
the product formula becomes  
\begin{equation}
([z+\I]_{aq^{-2\I}, bq^{-\I};q,p})^n =
\sum_{k=0}^n r_{n-k}^{\J}(a,b;q,p;B_{\I, \J, n})
\prod_{j=1}^k [z-\J (j-1)]_{aq^{2\J (j-1)},bq^{\J( j -1)};q,p}.
\end{equation}
If we define $\tilde{S}_{n,k}^{\I, \J}(a,b;q,p):=
r_{n-k}^{\J}(a,b;q,p;B_{\I, \J, n})$, then up to whether there is a rook
or not in the last column of $B_{\I, \J, n}$, we get the following recursion 
\begin{align}\label{eqn:genstirrec}
\tilde{S}_{n+1,k}^{\I, \J}(a,b;q,p)={}\notag&
W_{aq^{-2\I},bq^{-\I};q,p}(\I+(k-1)\J)\tilde{S}_{n, k-1}^{\I, \J}(a,b;q,p)\\
&+[\I +k\J]_{aq^{-2\I},bq^{-\I};q,p} \tilde{S}_{n,k}^{\I, \J}(a,b;q,p).
\end{align}
With the initial conditions $\tilde{S}_{0,0}^{\I, \J}(a,b;q,p)=1$ and 
$\tilde{S}_{n,k}^{\I, \J}(a,b;q,p)=0$ for $k<0$ or $k>n$,
\eqref{eqn:genstirrec} can be used to characterize $\tilde{S}_{n,k}^{\I, \J}(a,b;q,p)$.
In \cite{RW}, Remmel and Wachs developed a combinatorial interpretation for
$S_{n,k}^{\I, \J}(q)$ in terms of permutation statistics, colored partitions
and restricted growth functions. We can modify their $q$-weight
function to give a combinatorial interpretation for
$S_{n,k}^{\I,\J}(a,b;q,p)$ where 
$$\tilde{S}_{n,k}^{\I, \J}(a,b;q,p)=
\left(\prod_{j=1}^k W_{aq^{-2\I },bq^{-\I};q,p}(\I+(j-1)\J)\right)
S_{n,k}^{\I,\J}(a,b;q,p).$$
We shall assume that $0\le \I\le \J$. Let $\mathcal{CP}$ be the
collection of all set partitions of $\{ 0, 1, \dots, n\}$ whose
nonzero elements are colored with colors in the set
$\{ 0, 1,\dots, j-1\}$. We refer to the block of a colored partition
that contains $0$ as the \emph{zero-block}.
Define $\mathcal{CP}_{n,k}^{\I, \J}$ to be the subset of
$\mathcal{CP}$ consisting of partitions with $k+1$ blocks where the
elements are colored so that 
\begin{itemize}
\item[(a)] the nonzero elements of the zero-block have colors in
$\{ 0, \dots, i-1\}$, and 
\item[(b)] the smallest element of each block other than the
zero-block has color $0$. 
\end{itemize}
Note that there is a natural way to encode the set partitions of
$[n]$ as restricted growth functions. A \emph{restricted growth function}
is a word $w_1 \cdots w_n$ over the alphabet $[n]$ such that $w_1=1$ and
for $s=2,\dots, n$, we have $w_s \le 1+max\{ w_1,\dots, w_{s-1}\}$.
To a partition $\pi =\langle \pi_1,\dots, \pi_k\rangle$, where
$min(\pi_1)<\cdots <min(\pi_k)$, we associate the restricted
growth function $w_1 w_2\cdots w_n$, where $w_s= t$ if $s\in \pi_t$. 

Now we generalize this encoding to colored partitions.
Let $\pi=\langle \pi_0 ,\dots, \pi_k\rangle\in \mathcal{CP}_{n,k}^{\I,\J}$
where $min(\pi_0)<\cdots < min(\pi_k)$ and let $w(\pi)=w_0 w_1\cdots w_n$
where for all $0\le s \le n$, $w_s =t $ if $s\in \pi_t$.
Then we color $w$ with the same color that $s$ was colored with in $\pi$.
For example, if
$\pi=\langle \{0, 1^1, 4^0\}, \{ 2^0, 5^1\}, \{3^0, 6^2\}
\rangle \in \mathcal{CP}_{6,2}^{2,3}$ (here the exponents of the elements
are the respective colors), then
$w(\pi)=0 0^1 1^0 2^0 0^0 1^12^2$. We let
$\mathcal{RG}_{n,k}^{\I, \J}=\{ w(\pi)=
w_0 w_1 ^{e_1}\cdots w_n ^{e_n} \mid \pi\in \mathcal{CP}_{n,k}^{\I,\J}\}$.
We also express the colored word
$w=w_0 w_1 ^{e_1}\cdots w_n ^{e_n} \in \mathcal{RG}_{n,k}^{\I, \J}$ as a
pair of words $(w_0 w_1 \cdots w_n : e_1 \cdots e_n)$. Remmel and
Wachs~\cite[Theorem 18]{RW} showed that
$S_{n,k}^{\I, \J}(1)=|\mathcal{RG}_{n,k}^{\I, \J}|=|\mathcal{CP}_{n,k}^{\I, \J}|$
by constructing a bijection 
$\phi:\mathcal{RG}_{n,k}^{\I,\J}\rightarrow \mathcal{N}_{n-k}^{\J}(B_{\I,\J,n})$
as follows.
Let $(w:e)\in \mathcal{RG}_{n,k}^{\I, \J}$. We place rooks from left to right
column by column so that in column $s$, a rook is placed in the
$(\I +w_s \J -e_s)$-th available cell (that is not $\J$-attacked)
from the bottom. If no such cell is available then we leave the column
$s$ empty. For example, $(w:e)=(0012012 : 100012)$ corresponds to the
rook placement in Figure \ref{fig:we}. In Figure \ref{fig:we}, the
$\J$-attacked cells are denoted by $\bullet$'s and the uncancelled
cells above rooks are denoted by $\circ$'s.
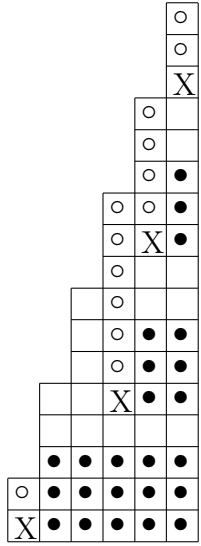
\begin{figure}[ht]
\begin{picture}(60,175)(0,0)
\multiput(0,10)(0,10){3}{\line(1,0){60}}
\multiput(10,40)(0,10){3}{\line(1,0){50}}
\multiput(20,70)(0,10){3}{\line(1,0){40}}
\multiput(30,100)(0,10){3}{\line(1,0){30}}
\multiput(40,130)(0,10){3}{\line(1,0){20}}
\multiput(50,160)(0,10){3}{\line(1,0){10}}
\multiput(0,10)(10,0){1}{\line(0,1){20}}
\multiput(10,10)(10,0){1}{\line(0,1){50}}
\multiput(20,10)(10,0){1}{\line(0,1){80}}
\multiput(30,10)(10,0){1}{\line(0,1){110}}
\multiput(40,10)(10,0){1}{\line(0,1){140}}
\multiput(50,10)(10,0){2}{\line(0,1){170}}
\put(2,11){X}
\put(32, 51){X}
\put(42, 101){X}
\put(52, 151){X}
\put(12, 13){$\bullet$}
\put(22, 13){$\bullet$}
\put(32, 13){$\bullet$}
\put(42, 13){$\bullet$}
\put(52, 13){$\bullet$}
\put(12, 23){$\bullet$}
\put(22, 23){$\bullet$}
\put(32, 23){$\bullet$}
\put(42, 23){$\bullet$}
\put(52, 23){$\bullet$}
\put(12, 33){$\bullet$}
\put(22, 33){$\bullet$}
\put(32, 33){$\bullet$}
\put(42, 33){$\bullet$}
\put(52, 33){$\bullet$}
\put(42, 53){$\bullet$}
\put(52, 53){$\bullet$}
\put(42, 63){$\bullet$}
\put(52, 63){$\bullet$}
\put(42, 73){$\bullet$}
\put(52, 73){$\bullet$}
\put(52, 103){$\bullet$}
\put(52, 113){$\bullet$}
\put(52, 123){$\bullet$}
\put(2, 23){$\circ$}
\put(32, 63){$\circ$}
\put(32, 73){$\circ$}
\put(32, 83){$\circ$}
\put(32, 93){$\circ$}
\put(32, 103){$\circ$}
\put(32, 113){$\circ$}
\put(42, 113){$\circ$}
\put(42, 123){$\circ$}
\put(42, 133){$\circ$}
\put(42, 143){$\circ$}
\put(52, 163){$\circ$}
\put(52, 173){$\circ$}
\end{picture}
\caption{The rook placement corresponding to
$(0012012 : 100012)\in \mathcal{RG}_{6,2}^{2,3}$.}\label{fig:we}
\end{figure}
We then define a $q$-statistic for $S_{q}^{\I, \J}(n,k)$.
For a given $\gamma =(w : e)\in \mathcal{RG}_{n,k}^{\I, \J}$,
let $m_s (\gamma)=max \{ w_1,\dots, w_{s-1}\}$, for each $s=1,\dots, n$,
and define 
\begin{align*}
\mathcal{MAX}&= \{s\in [n] : w_s > max \{w_0,\dots, w_{s-1}\}\},\\
inv(w:e)&=\J\sum_{1\le s < t\le n}\chi(w_s > w_t \text{ and } s\in
\mathcal{MAX}(w:e) )+\sum_{s=1}^n e_s.
\end{align*}
Then for 
$$D_{n,k}^{\I, \J}(q)=\sum_{\gamma\in\mathcal{RG}_{n,k}^{\I, \J} }q^{inv(\gamma)},$$
Remmel and Wachs showed in \cite[Theorem 19]{RW} that 
$D_{n,k}^{\I, \J}(q)=S_{n,k}^{\I, \J}(q)$. 

Now for $\gamma=(w:e)=(w_0 w_1\cdots w_n : e_1\cdots e_n)\in
\mathcal{RG}_{n,k}^{\I, \J}$, we define 
$$D_{n,k}^{\I, \J}(a,b;q,p)=\sum_{\gamma\in\mathcal{RG}_{n,k}^{\I, \J} }
\prod_{s=1}^n W_{aq^{-2\I},bq^{-\I};q,p}(\J |\{t<s : w_t > w_s
\text{ and } t\in \mathcal{MAX}(\gamma)\} |+e_s).$$
\begin{proposition}
For each $\gamma \in  \mathcal{RG}_{n,k}^{\I, \J}$, we have
$$wt^\J (\phi(\gamma))=\prod_{s=1}^n
W_{aq^{-2\I},bq^{-\I};q,p}(\J |\{t<s : w_t > w_s \text{ and }
t\in \mathcal{MAX}(\gamma)\} |+e_s),$$
and hence, 
$$D_{n,k}^{\I, \J}(a,b;q,p)=S_{n,k}^{\I, \J}(a,b;q,p).$$
\end{proposition}

\begin{proof}
Let $\gamma =(w:e)\in  \mathcal{RG}_{n,k}^{\I, \J}$.
Observe that for each $s=1,\dots, n$, column $s$ of $\phi(\gamma)$
has $\I +\J m_s (\gamma)$ cells that are not $\J$-attacked by any
rooks on the left, since there would be $(s-1-m_s(\gamma))$ many
rooks in the first $(s-1)$-columns, where
$m_s(\gamma)=max\{ w_0, w_1, \dots, w_{s-1}\}$. This implies that the
number of uncancelled cells above a rook in column $s$ is 
\begin{align*}
\I +\J m_s (\gamma) -(\I +\J w_s -e_s) &= \J (m_s (\gamma)-w_s)+e_s\\
&=\J |\{t<s : w_t > w_s \text{ and } t\in \mathcal{MAX}(\gamma)\} | +e_s.
\end{align*}
The weight of the top-most uncancelled cell is $w_{a,b;q,p}(1-\I)$,
and so the product of the weights of uncancelled cells in the column $s$ is 
\begin{align*}
&w_{a,b;q,p}(1-\I)w_{a,b;q,p}(2-\I)\cdots
w_{a,b;q,p}(\J |\{t<s : w_t > w_s \text{ and }
t\in \mathcal{MAX}(\gamma)\} | +e_s-\I)\\
={}& w_{aq^{-2\I },bq^{-\I};q,p}(1)\cdots
w_{aq^{-2\I },bq^{-\I };q,p}(\J |\{t<s : w_t > w_s \text{ and }
t\in \mathcal{MAX}(\gamma)\} | +e_s)\\
={}& W_{aq^{-2\I},bq^{-\I};q,p}(\J |\{t<s : w_t > w_s \text{ and }
t\in \mathcal{MAX}(\gamma)\} |+e_s).
\end{align*}
Thus $wt^\J (\phi(\gamma))$ is obtained by multiplying the above weights
over all $s=1,\dots, n$.
\end{proof}


\section{Elliptic file numbers}\label{sec:file}


In this section, we consider an elliptic analogue of the file numbers.
The file numbers and their $q$-analogue were first considered by
Garsia and Remmel in 1984 upon the introduction of
$q$-rook numbers in \cite{GR}
(but did not get into the final version of their paper).
The first time they actually appeared in literature under the name
``file numbers'' was in \cite{RW} where already $\mathfrak p,q$-extensions
were investigated.
Other instances where they appear include \cite{BCHR}
and \cite{McR}.

Given a board $B\subset [n]\times \mathbb N$, let  $\mathcal{F}_k (B)$
be the set of placements $Q$ of $k$ rooks in $B$ such that no two rooks
in $Q$ lie in the same column. 
We refer to such a $Q$ as a
\emph{file placement} of $k$ rooks in $B$. Thus in a file placement
$Q$, we do allow the possibility that two rooks lie in the same row.
Given a placement $Q\in \mathcal{F}_k (B)$, we let each rook in
$Q$ cancel all the cells below it in $B$. Let $u_B(Q)$ be the number
of cells in $B-Q$ which are not cancelled by any rook in $Q$.
Then the $q$-file numbers are defined by 
\begin{equation}
f_k(q;B)=\sum_{Q\in\mathcal{F}_k (B)}q^{u_B (Q)}.
\end{equation}
Garsia and Remmel proved the following product formula
involving the $q$-file numbers.

\begin{theorem}
For any skyline board $B=B(c_1, \dots, c_n)$, 
\begin{equation}\label{eqn:file}
\prod_{i=1}^n [z+c_i]_q = \sum_{k=0}^n f_{n-k}(q;B)([z]_q)^k.
\end{equation}
\end{theorem} 
This product formula can be proved by computing the sum
$$\sum_{Q \in \mathcal{F}_n (B_z)}q^{u_{B_z}(Q)},$$
where $B_z$ again denotes the extended board obtained by
attaching the board $[n]\times[z]$ below $B$,
in two different ways. We omit the details since we will prove
the elliptic extension of this result in Theorem~\ref{thm:elptfile}.

By distinguishing the cases whether there is a rook or not 
in the last column, one can also obtain the following recursion. 
\begin{proposition}\label{thm:recurfile}
Let $B$ be a skyline board and let $B\cup m$
denote the board obtained by adding a column of length $m$ to $B$.
Then for any nonnegative integer $k$ we have 
\begin{equation}
f_k(q;B\cup m)=q^{m}f_k(q;B)+[m]_q\,f_{k-1}(q;B).
\end{equation}
\end{proposition}

We can define an elliptic analogue of the $q$-file numbers by
assuming the same rook cancellation as in the $q$-case where
elliptic weights are assigned to the uncancelled cells.

\begin{definition}
Given a skyline board $B=B(c_1,\dots, c_n)$, we define the elliptic
analogue of the $k$-th file number by
\begin{subequations}
\begin{equation}
f_k (a,b;q,p;B) =\sum_{Q\in \mathcal{F}_k(B)}wt_f (Q),
\end{equation}
with 
\begin{equation}
wt_f (Q)=\prod_{(i,j)\in U_B (Q)}w_{a,b;q,p}(1-j),
\end{equation}
\end{subequations}
where the elliptic weight $w_{a,b;q,p}(l)$ of an integer $l$ is
defined in \eqref{def:smallelpwt}.
\end{definition}
Note that in this case, the elliptic weight of a cell only depends
on its row coordinate. 

\begin{theorem}\label{thm:elptfile}
For any skyline board $B=B(c_1,\dots, c_n)$, we have 
\begin{equation}\label{eqn:elptfile}
\prod_{i=1}^n [z+c_i]_{aq^{-2c_i},bq^{-c_i};q,p}=
\sum_{k=0}^n f_{n-k}(a,b;q,p;B)([z]_{a, b;q,p})^k.
\end{equation}
\end{theorem}

\begin{proof}
As before, it suffices to prove the theorem for nonnegative integer values of $z$.
We fix an integer $z$ and consider the extended board $B_z$ by attaching an $[n]\times [z]$ board
below the board $B$ and consider the $n$-file placements
$\mathcal{F}_n (B_z)$ in $B_z$. Then \eqref{eqn:elptfile} can be
proved by computing the sum 
\begin{equation}\label{eqn:elptfilez}
\sum_{Q\in \mathcal{F}_n (B_z)}wt_f (Q)
\end{equation}
in two ways. The left-hand side of \eqref{eqn:elptfile} computes
the above sum by placing rooks column by column. Since the elliptic weight
used to define $wt_f (Q)$ does not depend on the column coordinate of the
uncancelled cells, the weight sum in \eqref{eqn:elptfilez} is
the product of the weight sums coming from the possible placements in
each column, which is exactly the left-hand side of \eqref{eqn:elptfile}.
The right-hand side computes \eqref{eqn:elptfilez} by considering the
file placements in $B$ and in the extended part separately. 
\end{proof}

We have the following recursion for the elliptic file numbers,
which is an elliptic extension of Proposition~\ref{thm:recurfile}.
\begin{theorem}\label{thm:elptrecurfile}
Let $B$ be a skyline board, 
and $B\cup m$ denote the board obtained by adding a
column of height $m$ to $B$. Then, for any integer $k$, we have 
\begin{subequations}
\begin{align}
f_k (a,b;q,p;B)={}&0\qquad\text{for $k<0$,}\\
f_0 (a,b;q,p;B)={}&1\qquad\text{for
$B$ being the empty board},\\\intertext{and}
f_k (a,b;q,p;B\cup m)=
{}&W_{aq^{-2m},bq^{-m};q,p}(m)\,f_{k}(a, b;q,p;B)\notag\\\label{recrkf}
&+[m]_{aq^{-2m},bq^{-m};q,p}\,f_{k-1}(a, b;q,p;B).
\end{align}
\end{subequations}
\end{theorem}
\begin{proof}
This recursion stems from a weighted enumeration of a file placement of
$k$ rooks on $B\cup m$. 
We distinguish the cases whether there is a rook
in the last column or not.
The first term on the right-hand side of \eqref{recrkf} is obtained
when there is no rook in the last column.
The 
weight multiplied in front of $f_k(a,b;q,p;B)$ comes from the uncancelled $m$ cells
in the last column. The second term on the right-hand side of
\eqref{recrkf} is obtained when there
is a rook in the last column. The coefficient in front of
$f_{k-1}$ is a consequence of Lemma~\ref{lem:prod}.
\end{proof}

\begin{remark}
It is analytically and combinatorially obvious that two skyline boards
$B_1$ and $B_2$ for which the left-hand sides of
\eqref{eqn:elptfile} are equal ($B_2$ then must consist of the
columns of $B_1$ which may be permuted) have the
same elliptic file numbers. In this case we may refer to
such boards $B_1$ and $B_2$ as \emph{file equivalent}.
\end{remark}

\subsection{Elliptic Stirling numbers of the first kind}

For the staircase board $\mathsf{St}_n=B(0, 1, \dots, n-1)$,
the product formula in Theorem~\ref{thm:elptfile} becomes 
\begin{equation}\label{eqn:elptstirling1}
\prod_{i=1}^n [z+i-1]_{aq^{2(1-i)},bq^{1-i};q,p}=
\sum_{k=0}^n f_{n-k}(a,b;q,p;B)([z]_{a, b;q,p})^k.
\end{equation}
The file numbers $f_{n-k}(a,b;q,p;\mathsf{St}_n)$ are in fact the unsigned
\emph{elliptic Stirling numbers of the first kind}
$\mathfrak{c}_{a,b;q,p}(n,k)$ which have recently been
defined and studied (in a different setting) by
Zs\'ofia Keresk\'{e}nyin\'{e} Balogh
and the first author~\cite{KBSchl}.
For a bijection of file placements of $n-k$ rooks in $\mathsf{St}_n$
and permutations of $[n]$ with $k$ cycles, see the subsequent subsection,
where we consider a refinement of the Stirling numbers of the first kind.

By using the $(a,b,y,z)\mapsto(aq^{-2n},bq^{-n},n,z+n)$
case of the elementary identity \eqref{recellny},
or by distinguishing whether there is a rook or not in the last column,
we obtain from \eqref{eqn:elptstirling1}
the following recurrence relation
\begin{equation}\label{eqn:elptstir1rec}
\mathfrak{c}_{a,b;q,p}(n+1,k)=
[n]_{aq^{-2n}, bq^{-n};q,p}\mathfrak{c}_{a, b;q,p}(n,k)+
W_{aq^{-2n},bq^{-n};q,p}(n)\mathfrak{c}_{a,b;q,p}(n,k-1).
\end{equation}
With the conditions $\mathfrak{c}_{a,b;q,p}(0,0)=1$ and
$\mathfrak{c}_{a,b;q,p}(n,k)=0$  for $k<0$ or $k>n$,
the recurrence relation \eqref{eqn:elptstir1rec}
uniquely determines $\mathfrak{c}_{a,b;q,p}(n,k)$.

\subsection{Elliptic $r$-restricted Stirling numbers of the first kind}

The \emph{$r$-restricted (signless) Stirling numbers of the first kind}
(these are usually called $r$-Stirling numbers of the first kind
but we adopt the terminology from
\cite[see the sequences A143491, A143492 and A143493]{Sl}
to avoid possible confusion with the $q$-Stirling numbers), which we
denote by ${\mathfrak c}^{(r)} (n,k)$, are defined, for all positive $r$,
by the number of permutations of the set $\{ 1,\dots, n\}$ having $k$
cycles, such that the numbers $1, 2,\dots, r$ are in distinct cycles.
For $r=1$ (or $r=0$) they reduce to the usual Stirling numbers
of the first kind.
They are treated with some detail in \cite{B}, where it
is shown that the $r$-restricted Stirling numbers
of the first kind have the following generating function 
\begin{equation}
\sum_{k=0}^n {\mathfrak c}^{(r)}(n,k)z^k =
\left\{ \begin{array}{ll} z^r(z+r)(z+r+1)\cdots (z+n-1),&
n\ge r\ge 0,\\ 0,& \text{otherwise.}\end{array}\right.
\end{equation}
As in subsection~\ref{subsec:Str2}, let $\mathsf{St}_n^{(r)}$ denote the board
$\mathsf{St}_n^{(r)}=B(c_1,\dots, c_n)$ such that $c_i=0$ for $i=1,\dots,r$ and 
$c_i=i-1$, for $i=r+1,\dots,n$.
Then for $B=\mathsf{St}_n^{(r)}$ and $q\to 1$ the
product formula \eqref{eqn:file} becomes
\begin{equation}
z^r(z+r)(z+r+1)\cdots (z+n-1)=\sum_{k=0}^n f_{n-k}(1;\mathsf{St}_n^{(r)})z^k,
\end{equation}
which is the generating function for the $r$-restricted Stirling numbers
of the first kind. Thus we can identify ${\mathfrak c}^{(r)}(n,k)$
with $f_{n-k}(1;\mathsf{St}_n^{(r)})$. We can construct a bijection between file
placements of $n-k$ rooks in $\mathsf{St}_n^{(r)}$ and permutations
of $n$ numbers
with $k$ cycles, such that the numbers $1,2,\dots,r$ are in
distinct cycles. We start from the right-most rook.
If this rook is placed in the cell $(\alpha, l)$, then place $\alpha$
to the left of $l$ in the cycle notation. Then move to the second rook
to the left. If this second rook is in the cell $(\beta, l)$,
then put $\beta$ to the left of $\alpha$ in the same cycle ending
with $l$, but if this rook is in $(\beta, k)$ for $k\neq l$, then construct a
new cycle with $\beta$ and $k$ and put $\beta$ to the left of $k$.
While iterating this procedure, if there is a rook in $(l, h)$,
then append $h$ to the right-most place of the cycle ending with $l$.
The numbers which
never occurred in the cycle construction after reading all the rooks
in the file placement are fix-points, i.e., cycles of one element.
For example, given
the file placement in Figure~\ref{fig:file}, the right-most rook
gives $(8~ 3)$ in the cycle notation, and the second right-most rook
gives $(7~4)$. The third right-most rook puts $6$ to the left of $7$
and gives $(6~7~4)$. The fourth rook gives $(5~2)$ and the last rook
puts $1$ to the right of the cycle ending with $4$ and gives $(6~7~4~1)$.
Thus the permutation of $1,\dots, 8$ corresponding to the given file
placement in Figure~\ref{fig:file} is $(6~7~4~1)(5~2)(8~3)$.

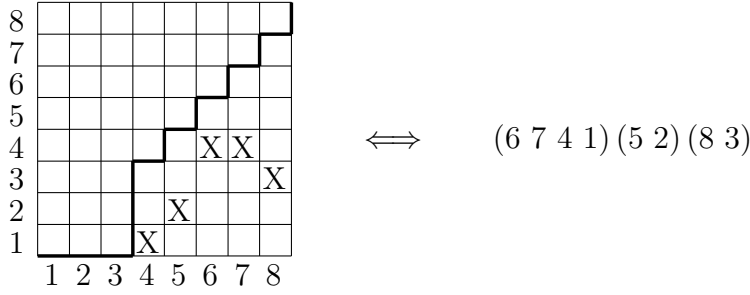
\begin{figure}[ht]
\begin{picture}(90,96)(0,0)
\multiput(10,10)(0,10){9}{\line(1,0){80}}
\multiput(10,10)(10,0){9}{\line(0,1){80}}
\thicklines\linethickness{1.3pt}
\multiput(10, 10)(10,0){1}{\line(1,0){30}}
\multiput(40, 10)(10,0){1}{\line(0,1){30}}
\multiput(40, 40)(10,0){1}{\line(1,0){10}}
\multiput(50, 40)(10,0){1}{\line(0,1){10}}
\multiput(50, 50)(10,0){1}{\line(1,0){10}}
\multiput(60, 50)(10,0){1}{\line(0,1){10}}
\multiput(60, 60)(10,0){1}{\line(1,0){10}}
\multiput(70, 60)(10,0){1}{\line(0,1){10}}
\multiput(70,70)(10,0){1}{\line(1,0){10}}
\multiput(80, 70)(10,0){1}{\line(0,1){10}}
\multiput(80,80)(10,0){1}{\line(1,0){10}}
\multiput(90, 80)(10,0){1}{\line(0,1){10}}
\put(12,1){$1$}
\put(22,1){$2$}
\put(32,1){$3$}
\put(42,1){$4$}
\put(52,1){$5$}
\put(62,1){$6$}
\put(72,1){$7$}
\put(82,1){$8$}
\put(1,11){$1$}
\put(1,21){$2$}
\put(1,31){$3$}
\put(1,41){$4$}
\put(1,51){$5$}
\put(1,61){$6$}
\put(1,71){$7$}
\put(1,81){$8$}
\put(41,11){X}
\put(51,21){X}
\put(61,41){X}
\put(71,41){X}
\put(81, 31){X}
\end{picture}
\raisebox{4.5em}{$\qquad\Longleftrightarrow\qquad(6~7~4~1)\,(5~2)\,(8~3)$}
\caption{A file placement with $r=3$, $n=8$, $n-k=5$,
and corresponding permutation in cycle notation.}\label{fig:file}
\end{figure}

Conversely, given a permutation in a cycle notation, we construct a
file placement as follows. In each cycle, put the smallest number at
the end of cycle. Then start from the left-most number, say $\alpha$,
find the left-most number to the right of $\alpha$ which is smaller
than $\alpha$, say $\beta$. Then this places a rook in the cell
$(\alpha, \beta)$. Continue this procedure to the right. For example,
given a permutation $(6~7~4~1)(5~2)(8~3)$, we place rooks in $(6,4)$,
$(7,4)$, $(4,1)$, $(5,2)$ and $(8,3)$, which recovers the file placement
in Figure~\ref{fig:file}. 

We can define an elliptic analogue of the $r$-restricted Stirling numbers
of the
first kind by using the product formula for the elliptic file numbers
with board $\mathsf{St}_n^{(r)}$ as the generating function :
\begin{equation}
( [z]_{a,b;q,p})^r \prod_{i=1}^{n-r}
[z+r+i-1]_{aq^{2(1-i-r)},bq^{1-i-r};q,p}=
\sum_{k=0}^n f_{n-k}(a,b;q,p;\mathsf{St}_n^{(r)})([z]_{a,b;q,p})^k.
\end{equation}
Let ${\mathfrak c}^{(r)}_{a,b;q,p}(n,k)$ denote
$f_{n-k}(a,b;q,p;\mathsf{St}_n^{(r)})$.
Then, by distinguishing whether there is a rook or not in the last column,
we can deduce the recurrence relation of
${\mathfrak c}^{(r)}_{a,b;q,p}(n,k)$, namely, for $k\ge r-1$, 
\begin{align}
{\mathfrak c}_{a,b;q,p}^{(r)}(n+1,k)&=
[n]_{aq^{-2n},bq^{-n};q,p}{\mathfrak c}^{(r)}_{a,b;q,p}(n,k)+
W_{aq^{-2n},bq^{-n};q,p}(n){\mathfrak c}^{(r)}_{a,b;q,p}(n,k-1).
\end{align}
This recursion uniquely determines ${\mathfrak c}^{(r)}_{a,b;q,p}(n,k)$
with the conditions 
\begin{align}
\mathfrak c^{(r)}_{a,b;q,p}(n,k)&=0\qquad\text{for $k<r-1$ or $k>n$},\notag\\
\mathfrak c^{(r)}_{a,b;q,p}(r-1,r-1)&=1.\notag
\end{align}

\subsection{Abel boards and weighted forests}

Let $\mathsf A_n$ denote the \emph{Abel board}, the $[n-1]\times [n]$ board with
column heights $(0,n,\dots, n)$. For $B=\mathsf A_n$, the product formula
involving the file numbers \eqref{eqn:file}, when $q\to 1$, becomes 
$$z(z+n)^{n-1}=\sum_{k=0}^n f_{n-k}(1;\mathsf A_n)z^k.$$
These polynomials are a special case of the general Abel polynomials
$z(z+\alpha n)^{n-1}$, which we consider separately in the discussion following
equation~\eqref{acpol}. 
The coefficient $f_{n-k}(1;\mathsf A_n)=t_{n,k}=
\binom{n-1}{k-1}n^{n-k}$ counts the number
of labeled forests on $n$ vertices composed of $k$ rooted trees \cite{MR}.
Goldman and Haglund explained this equality bijectively in \cite{GH}.
More precisely, they construct a bijection between the set
$$
R_{n,k}=\{(Q, u), ~u\in \{1,2,\dots, n\}\},
$$
where $Q$ is a file placement
of $n-k$ rooks on $\mathsf A_n$, and
$$
F_{n,k}=\{\text{marked rooted forests of $k$
rooted trees on $n$ labeled vertices}\},
$$
where a \emph{marked rooted forest}
is a forest of rooted trees with one distinguished vertex in the forest
(the mark), via constructing bijections between $R_{n,k}$ and $Q_{n,k}$,
and between $Q_{n,k}$ and $F_{n,k}$, where $Q_{n,k}$ is a class of ``marked''
partial endofunctions. 
Thus Goldman and Haglund actually proved $nf_{n-k}(1;\mathsf A_n)=nt_{n,k}$.
We present a new proof which
establishes $f_{n-k}(1;\mathsf A_n)=t_{n,k}$ directly.
In the following, we describe a
bijection between
$$
\mathcal R_{n,k}=
\{\text{file placements of $n-k$ rooks on $\mathsf A_n$}\}
$$
and
$$
\mathcal F_{n,k}=\{\text{rooted forests of $k$
rooted trees on $n$ labeled vertices}\}.
$$
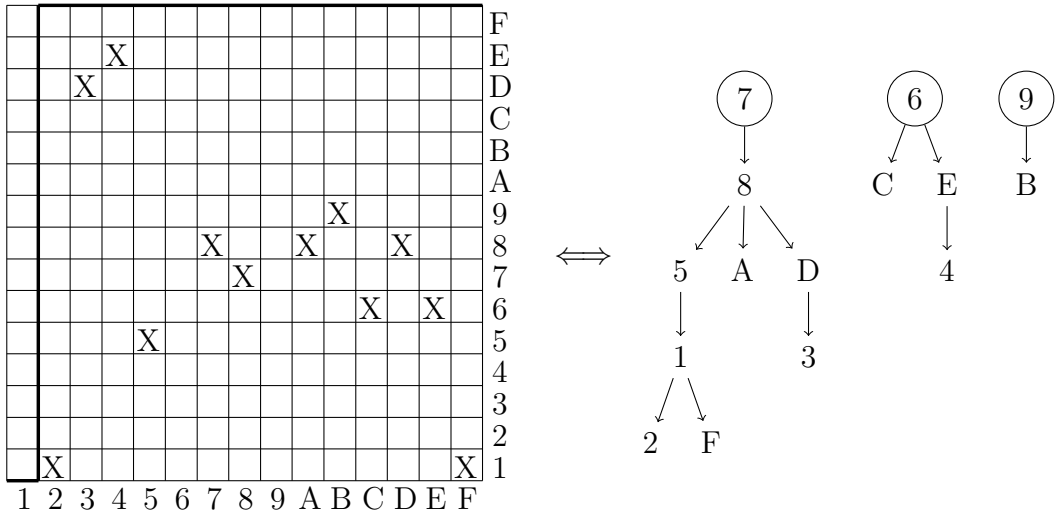
\begin{figure}[ht]
\begin{picture}(150,170)(0,0)
\multiput(10,10)(0,10){16}{\line(1,0){150}}
\multiput(10,10)(10,0){16}{\line(0,1){150}}
\thicklines\linethickness{1.3pt}
\multiput(10,10)(10,0){1}{\line(1,0){10}}
\multiput(20,10)(10,0){1}{\line(0,1){150}}
\multiput(20,160)(10,0){1}{\line(1,0){140}}
\put(13,1){1}
\put(23,1){2}
\put(33,1){3}
\put(43,1){4}
\put(53,1){5}
\put(63,1){6}
\put(73,1){7}
\put(83,1){8}
\put(93,1){9}
\put(102,1){A}
\put(112,1){B}
\put(122,1){C}
\put(132,1){D}
\put(142,1){E}
\put(152,1){F}
\put(163,11){1}
\put(163,21){2}
\put(163,31){3}
\put(163,41){4}
\put(163,51){5}
\put(163,61){6}
\put(163,71){7}
\put(163,81){8}
\put(163,91){9}
\put(162,101){A}
\put(162,111){B}
\put(162,121){C}
\put(162,131){D}
\put(162,141){E}
\put(162,151){F}
\put(21,11){X}
\put(151,11){X}
\put(51,51){X}
\put(121,61){X}
\put(141,61){X}
\put(81,71){X}
\put(71,81){X}
\put(101,81){X}
\put(131,81){X}
\put(111,91){X}
\put(31,131){X}
\put(41,141){X}
\end{picture}
\qquad\quad\raisebox{8em}{$\Longleftrightarrow\;$}
\raisebox{13em}{\begin{forest} baseline
[7,circle,draw [8,edge=-> [5,edge=-> [1,edge=-> [2,edge=->] [F,edge=->]]]
[A,edge=->] [D,edge=-> [3,edge=->]]]]
\end{forest}
$\;$
\begin{forest} baseline
[6,circle,draw [C,edge=->] [E,edge=-> [4,edge=->]]]
\end{forest}
$\;$
\begin{forest} baseline
[9,circle,draw [B,edge=->]]
\end{forest}}
\caption{A file placement with $n=15$, $k=3$,
and corresponding forest.}\label{fig:abel}
\end{figure}

The bijection is best described by considering an example,
see Figure~\ref{fig:abel}, where for nicer layout, we have labeled
the indices by hexadecimal digits. (Moreover, in the forest,
we have circled the labels of the roots of the trees.)
In this example, we have $n=15$ and $k=3$; we thus consider a $12$-file
placement of rooks on the $\mathsf A_{15}$ board.
We shall successively transform this file placement into a forest of
$15-12=3$ labeled trees.
The rooks are in positions $(2,1)$, $(F,1)$, $(5,5)$, $(C,6)$, $(E,6)$,
$(8,7)$, $(7,8)$, $(A,8)$, $(D,8)$, $(B,9)$, $(3,D)$, $(4,E)$, listed
from bottom to top, left to right.
Now we first identify the indices of the $k=3$ empty columns,
which are $1$, $6$, $9$. We let $1$ (which is special since the
first column is always empty) be the preliminary label of the
root of the first tree, while the other two numbers, $6$ and $9$,
be the labels of the respective roots of the second and third trees.
Note that these other roots shall not change whereas the first root can
change.
Next we look at the positions of all the rooks, first row-wise
from bottom to top and then from left to right within each row.
Now we interpret a rook in position $(i,j)$ as a directed edge
(forming a simple path) from vertex $j$ to vertex $i$.
After this we continue the path by looking for
rooks in positions $(l,i)$, in which case the path continues to
go from vertex $i$ to vertex $l$, etc.
We transitively collect all paths in the file placement
(in the prescribed order) and finally obtain an ordered collection
of maximal chains (of paths which cannot be continued),
and of disjoint cycles (where vertices already visited are reached again).
Numbers occurring in the cycles can also appear in the chains. 
In our example, we have the following seven chains and two cycles.

\smallskip
maximal chains:

$1\to 2$, $1\to F$, $6\to C$, $6\to E\to 4$,
$7\to 8\to A$, $7\to 8\to D\to 3$, $9\to B$.

\smallskip
cycles: $(5)$, $(7\,8)$.

\smallskip
It is important to insist that the minimal elements of the respective
cycles are listed first.
In the algorithm, the cycles are indeed obtained in the above order,
i.e., by left-to-right increasing order of their minimal elements.
Now, say we have obtained $l$ cycles, $\gamma_1,\dots,\gamma_l$, listed
by increasing order of their minimal elements.
We then reverse the order of the $l$ cycles,
i.e., write out $\gamma_l,\dots,\gamma_1$ in decreasing order
of their minimal elements while keeping the minimal
elements of each cycle at the first position.
In our example, the two cycles are thus relisted as
$(7\,8)$, $(5)$.

Now we form a new chain using all the labels from left to right
appearing in the complete list of cycles,
here $7\to8\to5$, and place this in the first tree before $1$ (i.e., we
also put an edge leading from the last vertex $5$ to $1$).
Hence, $7$ is now the new root of the first tree and we have a path
leading to $1$.
(After a short moment of reflection, this part of the 
correspondence is easily noticed to be reversible, since $5$ is the minimal
element before $1$, the minimal element before $5$
is $7$, etc., thus all the cycles can readily be determined.)

What remains to be done is to translate the maximal chains obtained
from the file placement to form trees in the forest.
This is done in the obvious way (and is clearly reversible); see
Figure~\ref{fig:abel} for the result.

Since there is at most one rook in each column of the file placement,
it is guaranteed that each vertex in the corresponding directed graph has
at most one predecessor, i.e.\ the resulting graph is indeed a forest.
Thus, we can conclude that each forest of $n$ labeled vertices composed of
$k$ components corresponds to exactly one $(n-k)$-rook file placement on the
board $\mathsf A_{n}$ and vice versa.

For the Abel board
$\mathsf A_n$, the product formula in Theorem~\ref{thm:elptfile} becomes 
\begin{equation}\label{prabel}
[z]_{a,b;q,p}([z+n]_{aq^{-2n},bq^{-n};q,p})^{n-1}=
\sum_{k=0}^n f_{n-k}(a,b;q,p;\mathsf A_n)([z]_{a,b;q,p})^k.
\end{equation}
As explained above, we can interpret the
coefficient $f_{n-k}(a,b;q,p;\mathsf A_n)$ as the weighted sum of
labeled forests on $n$ vertices composed of $k$ rooted trees.
In the above algorithm for obtaining the forest,
we weight the edge $j\to i$ by $\prod_{l=1}^{j-1}w_{a,b;q,p}(l)$
which in the file placement
corresponds to the product of the weights of uncancelled cells
above $(i,j)$. If there is an empty column containing no rooks,
then we weight the vertex corresponding to such a column by
$\prod_{l=1}^n w_{a,b;q,p}(l)$. 
This yields a weighted forest of $k$
rooted labeled trees on $n$ vertices corresponding to a given
$(n-k)$-file placement. 

The coefficients in \eqref{prabel} have a nice closed form
\begin{equation}\label{prabel1}
f_{n-k}(a,b;q,p;\mathsf A_n)=\binom{n-1}{k-1}
\left(W_{aq^{-2n},bq^{-n};q,p}(n)\right)^{k-1}
\left([n]_{aq^{-2n},bq^{-n};q,p}\right)^{n-k},
\end{equation}
which is easy to prove directly combinatorially, or by using the identity
$$
[z+n]_{aq^{-2n},bq^{-n}}=[n]_{aq^{-2n},bq^{-n}}+
W_{aq^{-2n},bq^{-n};q,p}(n)[z]_{aq^{-2n},bq^{-n}}
$$
(which is the $(y,z,a,b)\mapsto(n,z+n,aq^{-2n},bq^{-n})$ case of
Equation~\eqref{recellny}),
together with the classical binomial theorem.

The above bijection easily extends to the case of $r$-restricted Abel
boards $\mathsf A_n^{(r)}=B(0,\dots,0,n,\dots,n)$ of
$r$ columns of height zero and $n-r$ columns of height $n$.
Say, we consider a file placement on  $\mathsf A_n^{(r)}$.
Then the bijection transforms a file placement of $n-k$
rooks on this board to a forest of $k$ components where
the first $r$ numbers $1,2,\dots,r$ are in distinct trees and
the $r-1$ numbers $2,\dots,r$ are roots.
Now, by interchanging the labels $1$ and $r$ we immediately
obtain a forest of $n$ vertices of $k$ labeled trees, where
the first $r$ numbers $1,2,\dots,r$ are in distinct trees 
and where moreover the first $r-1$ numbers $1,2,\dots,r-1$ are all roots
(among the $k$ roots of the forest).
The number of such forests is 
$f_{n-k}(1;\mathsf A^{(r)}_n)=t^{(r)}_{n,k}=\binom{n-r}{k-r}n^{n-k}$.
Given the analogy to the $r$-restricted Stirling numbers of
the first and second kinds and of the $r$-restricted Lah numbers,
it seems appropriate to refer to $t_{n,k}$ as \emph{Abel numbers} and to
$t^{(r)}_{n,k}$ as \emph{$r$-restricted Abel numbers}.

For the $r$-restricted Abel board
$\mathsf A_n^{(r)}$, the product formula in Theorem~\ref{thm:elptfile} becomes 
\begin{equation}\label{prabelr}
([z]_{a,b;q,p})^r([z+n]_{aq^{-2n},bq^{-n};q,p})^{n-r}=
\sum_{k=r-1}^n f_{n-k}(a,b;q,p;A_n^{(r)})([z]_{a,b;q,p})^k.
\end{equation}

The coefficients in \eqref{prabelr} have a nice closed form
\begin{equation}\label{prabel1r}
f_{n-k}(a,b;q,p;\mathsf A_n^{(r)})=\binom{n-r}{k-r}
\left(W_{aq^{-2n},bq^{-n};q,p}(n)\right)^{k-r}
\left([n]_{aq^{-2n},bq^{-n};q,p}\right)^{n-k}.
\end{equation}

Lastly, we consider the general Abel board
$\mathsf A_{\alpha n,n}=B(0,\alpha n,\dots,\alpha n)=[n-1]\times[\alpha n]$ for a
positive integer $\alpha$.
The coefficients of the Abel polynomials
\begin{equation}\label{acpol}
z(z+\alpha n)^{n-1}=\sum_{k=0}^n f_{n-k}(1;\mathsf A_{\alpha n,n})z^k
\end{equation}
have a very  simple combinatorial interpretation.
They count the number of forests of $n$ labeled vertices composed of $k$ rooted
trees where each of the vertices can have one of $\alpha$ colors
(distinct vertices may have the same color),
where the $k$ roots must all have the first color.
This number of course is $\alpha^{n-k}$ times the usual monocolor case,
since each of the $n-k$ vertices which are not roots can assume one
of $\alpha$ colors. We therefore have
$f_{n-k}(1;\mathsf A_{\alpha n,n})=t_{\alpha,n,k}=
\binom{n-1}{k-1}(\alpha n)^{n-k}$.
This interpretation also easily comes out of the file placement model
as follows. If there is a rook in position $(i,(c-1)n+j)$
for $1\le i,j\le n$ and $1\le c\le \alpha$, then form a directed
path from $j$ to $i$ and assign color $c$ to the vertex $i$.

Even more generally, we could consider the Abel board
$\mathsf A_{\alpha n,n}$ with
$\alpha=\frac mn$, i.e., $\mathsf A_{m,n}=[n-1]\times[m]$ for a
positive integer $m$. Even here it is not difficult to give a combinatorial
interpretation (which is consistent with the file placement
model). The coefficients of the polynomials
\begin{equation}\label{mnforest}
z(z+m)^{n-1}=\sum_{k=0}^n f_{n-k}(1;\mathsf A_{m,n})z^k
\end{equation}
count the number of forests of $n$ labeled vertices composed of $k$ rooted
trees where each of the vertices can have one of $\lceil \frac mn\rceil$
colors (distinct vertices may have the same color),
the $k$ roots must all have the first color, but only
the successors of $1,2,\dots,m-\lfloor\frac{m-1}n\rfloor n$
are allowed to assume the highest color $\lceil \frac mn\rceil$.
Here $\lceil x\rceil:=\min\{y\in\mathbb Z:y\ge x\}$
and  $\lfloor x\rfloor:=\max\{y\in\mathbb Z:y\le x\}$ are the ceiling
and floor functions, respectively.

For instance, if $m=4$ and $n=3$, Equation~\eqref{mnforest} becomes
$z(z+4)^2=z^3+8z^2+16z$. Accordingly, using the vertices $1,2,3$ we
form colored forests containing exactly $k$ trees where
we may use two colors, say, black and white, to color the vertices
but only the successors of $1$ can be white. There is exactly
one forest containing three trees, each consisting of only a root.
There are $8$ such forests containing two trees, see Figure~\ref{fig:abelmn},
and there are $16$ such forests containing exactly one tree,
see Figure~\ref{fig:abelmn1},
where for convenience primed labels indicate
white vertices.

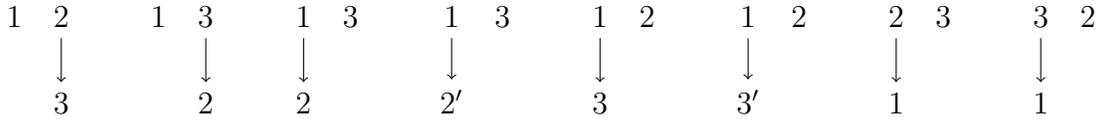
\begin{figure}[ht]
\begin{forest} baseline
[1]
\end{forest}
\begin{forest} baseline
[2 [3,edge=->]]
\end{forest}
\quad$\;$
\begin{forest} baseline
[1]
\end{forest}
\begin{forest} baseline
[3 [2,edge=->]]
\end{forest}
\quad$\;$
\begin{forest} baseline
[1 [2,edge=->]]
\end{forest}
\begin{forest} baseline
[3]
\end{forest}
\quad$\;$
\begin{forest} baseline
[1 [$2'$,edge=->]]
\end{forest}
\begin{forest} baseline
[3]
\end{forest}
\quad$\;$
\begin{forest} baseline
[1 [3,edge=->]]
\end{forest}
\begin{forest} baseline
[2]
\end{forest}
\quad$\;$
\begin{forest} baseline
[1 [$3'$,edge=->]]
\end{forest}
\begin{forest} baseline
[2]
\end{forest}
\quad$\;$
\begin{forest} baseline
[2 [1,edge=->]]
\end{forest}
\begin{forest} baseline
[3]
\end{forest}
\quad$\;$
\begin{forest} baseline
[3 [1,edge=->]]
\end{forest}
\begin{forest} baseline
[2]
\end{forest}
\caption{Colored forests of two trees corresponding to the board
$\mathsf A_{4,3}$.}
\label{fig:abelmn}
\end{figure}

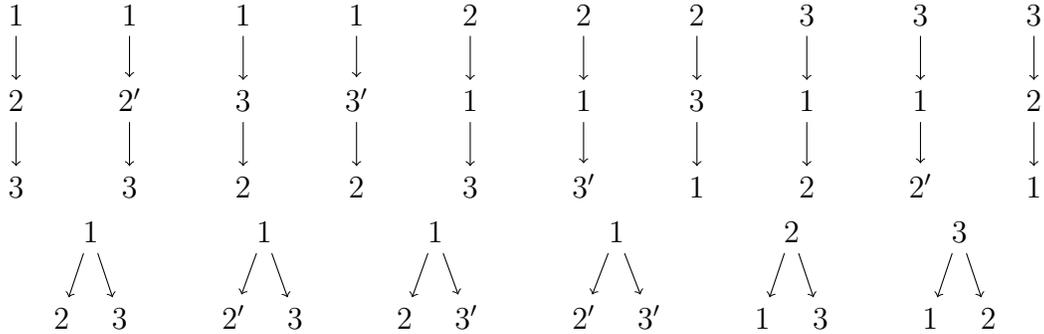
\begin{figure}[ht]
\begin{forest} baseline
[1 [2,edge=->[3,edge=->]]]
\end{forest}
\qquad
\begin{forest} baseline
[1 [$2'$,edge=->[3,edge=->]]]
\end{forest}
\qquad
\begin{forest} baseline
[1 [3,edge=->[2,edge=->]]]
\end{forest}
\qquad
\begin{forest} baseline
[1 [$3'$,edge=->[2,edge=->]]]
\end{forest}
\qquad
\begin{forest} baseline
[2 [1,edge=->[3,edge=->]]]
\end{forest}
\qquad
\begin{forest} baseline
[2 [1,edge=->[$3'$,edge=->]]]
\end{forest}
\qquad
\begin{forest} baseline
[2 [3,edge=->[1,edge=->]]]
\end{forest}
\qquad
\begin{forest} baseline
[3 [1,edge=->[2,edge=->]]]
\end{forest}
\qquad
\begin{forest} baseline
[3 [1,edge=->[$2'$,edge=->]]]
\end{forest}
\qquad
\begin{forest} baseline
[3 [2,edge=->[1,edge=->]]]
\end{forest}
\qquad
\begin{forest} baseline
[1 [2,edge=->] [3,edge=->]]
\end{forest}
\qquad
\begin{forest} baseline
[1 [$2'$,edge=->] [3,edge=->]]
\end{forest}
\qquad
\begin{forest} baseline
[1 [2,edge=->] [$3'$,edge=->]]
\end{forest}
\qquad
\begin{forest} baseline
[1 [$2'$,edge=->] [$3'$,edge=->]]
\end{forest}
\qquad
\begin{forest} baseline
[2 [1,edge=->] [3,edge=->]]
\end{forest}
\qquad
\begin{forest} baseline
[3 [1,edge=->] [2,edge=->]]
\end{forest}
\caption{Colored forests of one tree corresponding to the board
$\mathsf A_{4,3}$.}
\label{fig:abelmn1}
\end{figure}

The above generalized Abel case involving file placements on the board
$\mathsf A_{m,n}$
can even be extended to the $r$-restricted case where for a nonnegative
integer $r$ we consider the board $\mathsf A^{(r)}_{m,n}=[n-r]\times[m]$.
We have the same combinatorial interpretation for the colored forests
but with the additional restriction that the numbers $1,2,\dots,r$ are
in different trees and $1,2\dots,r-1$ are roots.
 
In the elliptic case, we have
\begin{subequations}
\begin{equation}\label{prabelrmell}
([z]_{a,b;q,p})^r([z+m]_{aq^{-2m},bq^{-m};q,p})^{n-r}=
\sum_{k=0}^n f_{n-k}(a,b;q,p;\mathsf A^{(r)}_{m,n})([z]_{a,b;q,p})^k
\end{equation}
and for the coefficients we have the explicit formula
\begin{equation}\label{prabel1m}
f_{n-k}(a,b;q,p;\mathsf A^{(r)}_{m,n})=
\binom{n-r}{k-r}\left(W_{aq^{-2m},bq^{-m};q,p}(m)\right)^{k-r}
\left([m]_{aq^{-2m},bq^{-m};q,p}\right)^{n-k}.
\end{equation}
\end{subequations}


\subsection{Elliptic analogue of generalized Stirling numbers
of the first kind}


Upon the introduction of the generalized $(\mathfrak p, q)$-Stirling numbers
of the second kind
(see Section \ref{subsec:pqstir} for the elliptic analogue of them), 
Remmel and Wachs \cite{RW} also introduced and studied the generalized
$(\mathfrak p, q)$-Stirling numbers 
of the first kind, denoted by $\mathfrak c_{n,k}^{\I, \J}(\mathfrak p, q)$.
In \cite{RW}, it is shown that $\mathfrak c_{n,k}^{\I, \J}(\mathfrak p, q)$
satisfy
$\mathfrak c_{0,0}^{\I, \J}(\mathfrak p, q)=1$,
$\mathfrak c_{n,k}^{\I, \J}(\mathfrak p, q)=0$
if $k<0$ or $k>n$, and 
$$\mathfrak c_{n+1,k}^{\I, \J}(\mathfrak p, q)
= \mathfrak c_{n,k-1}^{\I, \J}(\mathfrak p, q)+
[\I+n\J]_{\mathfrak p,q}\mathfrak c_{n,k}^{\I, \J}(\mathfrak p, q).$$
It is also shown that 
\begin{equation}
([z]_{\mathfrak p,q}+[\I]_{\mathfrak p,q})([z]_{\mathfrak p,q}+[\I+\J]_{\mathfrak p,q})
\cdots ([z]_{\mathfrak p,q}+[\I+(n-1)\J]_{\mathfrak p,q})
=\sum_{k=0}^n \mathfrak c_{n,k}^{\I, \J}(\mathfrak p, q)([z]_{\mathfrak p,q})^k.
\end{equation}
$\mathfrak c_{n,k}^{\I, \J}(\mathfrak p, q)$ can be computed as a
weighted sum of file placements in $B_{\I, \J,n}$. Recall that 
$B_{\I, \J,n}=B(\I, \I+\J, \dots, \I+(n-1)\J)$. 

Given a file placement
$Q\in \mathcal{F}_{k}(B)$, for $B$ a skyline board, 
define
$$w_{\mathfrak p, q,B}(Q)=q^{\alpha_B (Q)}\mathfrak p^{\beta_B(Q)},$$
where
\begin{align*}
\alpha_B (Q) & = \text{ the number of cells in $B$ that lie above
some rook $\mathbf r$ in $Q$},\\
\beta_B (Q) &= \text{ the number of cells in $B$ that lie below
some rook $\mathbf r$ in $Q$}.
\end{align*}
If we define 
$$\tilde{f}_k (\mathfrak p,q;B)=\sum_{Q\in \mathcal{F}_k (B)}w_{\mathfrak p, q,B}(Q),$$
then for $0\le k\le n$ it is shown in \cite{RW} that 
$$\mathfrak c_{n,k}^{\I, \J}(\mathfrak p, q)=
\tilde{f}_{n-k} (\mathfrak p,q;B_{\I,\J,n}).$$

We now establish an elliptic analogue of
$\mathfrak c_{n,k}^{\I, \J}(\mathfrak p, q)$
by modifying the weight function 
$w_{\mathfrak p, q,B}(Q)$ to an elliptic function. Given a skyline board $B$,
define 
$$\tilde{f}_k (a,b;q,p;B)=\sum_{Q\in\mathcal{F}_k(B)}\widetilde{wt}_f (Q),$$
where 
\begin{equation}\label{eqn:filewt'}
\widetilde{wt}_f (Q) =\prod_{(i,j)\in\mathcal{A}_B (Q)}w_{a,b;q,p}(i-j),
\end{equation}
and $\mathcal{A}_B (Q)$ is the set of cells in $B$ that lie above some rook
$\mathbf r$ in $Q$. 

\begin{proposition}\label{prop:ellfile}
For a skyline board $B=B(c_1,\dots, c_n)$, we have  
\begin{equation}\label{eqn:ellfile2}
\prod_{i=1}^n \left([z]_{a,b;q,p}+[c_i]_{aq^{2(i-1-c_i)},bq^{i-1-c_i};q,p} \right)=
\sum_{k=0}^n \tilde{f}_{n-k}(a,b;q,p;B)([z]_{a,b;q,p})^k.
\end{equation}
\end{proposition}
\begin{proof}
We consider the extended board $B_z$ and compute the sum 
$$
\sum_{P\in \mathcal{F}_n (B_z)}\overline{wt}_f (P)
$$
in two different ways, where 
$$
\overline{wt}_f (P) = \prod_{(i,j)\in P}\overline{w}_{a,b;q,p;B_z}(i,j),
$$
$$\overline{w}_{a,b;q,p;B_z}(i,j) = \left\{
\begin{array}{ll}
W_{aq^{2(i-1-c_i)},bq^{i-1-c_i};q,p}(c_i-j), & \text{ if } (i,j)\in B,\\
W_{a,b;q,p}(-j),& \text{ if $(i,j)$ is below the ground.} 
\end{array}
\right.$$

Recall that the line separating the board $B$ and the extended part of
$B_z$ is called the ground, and the 
row coordinates below the ground are $0,-1,-2,\dots, 1-z$, from top to bottom. 
To obtain the left-hand side of \eqref{eqn:ellfile2}, we place
$n$ rooks column by column. 
In $i$-th column, possible placements of a rook above the ground contribute 
$$1+W_{aq^{2(i-1-c_i)},bq^{i-1-c_i};q,p}(1)+\cdots + W_{aq^{2(i-1-c_i)},bq^{i-1-c_i};q,p}(c_i-1)=
[c_i]_{aq^{2(i-1-c_i)},bq^{i-1-c_i};q,p}$$
and possible placements below the ground give 
$$1+W_{a,b;q,p}(1)+\cdots + W_{a,b;q,p}(z-1)=[z]_{a,b;q,p}.$$
Sum of $[c_i]_{aq^{2(i-1-c_i)},bq^{i-1-c_i};q,p}$ and $[z]_{a,b;q,p}$ gives the $i$-th factor 
in the left-hand side of \eqref{eqn:ellfile2}. 

To get the right-hand side of \eqref{eqn:ellfile2}, we start with a
file placement $Q\in\mathcal{F}_{n-k}(B)$ 
and extend it to a file placement of $n$ rooks by placing $k$ rooks below
the ground.
In each empty column, possible placements of a rook below the ground give
$[z]_{a,b;q,p}$,
as computed above, hence placements of $k$ rooks below the ground will
give the factor $([z]_{a,b;q,p})^k$.
Note that $\overline{w}_{a,b;q,p;B_z}(i,j) =W_{aq^{2(i-1-c_i)},bq^{i-1-c_i};q,p}(c_i-j)$
for $(i,j)\in B$ is 
defined so that $\overline{w}_{a,b;q,p;B_z}(i,j)$ equals to
$\prod_{t=j+1}^{c_i}w_{a,b;q,p}(i-t)$ for $(i,t)$ 
being the coordinates of the cells above the rook in $(i,j)$. Thus,
\begin{align*}
\sum_{P\in \mathcal{F}_n (B_z)}\overline{wt}_f (P) &=
\sum_{k=0}^n \left(\sum_{Q\in\mathcal F_{n-k}(B)}
\prod_{(i,j)\in Q}\overline{w}_{a,b;q,p;B}(i,j) \right)  ([z]_{a,b;q,p})^k\\
&= \sum_{k=0}^n \tilde{f}_{n-k}(a,b;q,p;B)([z]_{a,b;q,p})^k.\\[-3.6em]
\end{align*}
\end{proof}

\smallskip
If we apply Proposition \ref{prop:ellfile} to the board $B_{\I, \J, n}$,
then we get 
\begin{align*}
\prod_{s=1}^n \left([z]_{a,b;q,p}+
[\I +(s-1)\J]_{aq^{-2(\I +(s-1)(\J-1))},bq^{-(\I +(s-1)(\J-1))};q,p} \right)&\\
=\sum_{k=0}^n \tilde{f}_{n-k}(a,b;q,p;B_{\I, \J, n})([z]_{a,b;q,p})^k&.
\end{align*}
We can define $\mathfrak c_{n,k}^{\I, \J}(a,b;q,p):=
\tilde{f}_{n-k}(a,b;q,p;B_{\I, \J, n})$ which is an elliptic analogue of
$\mathfrak c_{n,k}^{\I, \J}(\mathfrak p, q)$. 
Then, by considering whether there is a rook or not in the last
column of $B_{\I, \J, n}$,  we get the recurrence relation 
$$\mathfrak c_{n+1,k}^{\I, \J}(a,b;q,p)=\mathfrak c_{n,k-1}^{\I, \J}(a,b;q,p)+
[\I +n\J]_{aq^{-2(\I +n(\J-1))},bq^{-(\I +n(\J-1))};q,p}
\mathfrak c_{n,k}^{\I, \J}(a,b;q,p)$$
which determines $\mathfrak c_{n,k}^{\I, \J}(a,b;q,p)$ with the conditions 
$$\mathfrak c_{0,0}^{\I, \J}(a,b;q,p)=1, \text{ and } \quad
\mathfrak c_{n,k}^{\I, \J}(a,b;q,p)=0 \text{ if $k<0$ or $k>n$}.$$

\begin{remark}
In \cite{RW}, Remmel and Wachs showed that for
$s_{n,k}^{\I, \J}(\mathfrak p,q)=
(-1)^{n-k}\mathfrak c_{n,k}^{\I, \J}(\mathfrak p,q)$, the
two matrices $||s_{n,k}^{\I, \J}(\mathfrak p,q)||$ and
$||S_{n,k}^{\I, \J}(\mathfrak p,q)||$ 
are inverses of each other (we refer to Section~\ref{subsec:pqstir}
for the definition
of $S_{n,k}^{\I, \J}(\mathfrak p,q)$), namely, for all $0\le r \le n$, 
\begin{equation}\label{eqn:Stirlingmtx}
\sum_{k=r}^n S_{n,k}^{\I, \J}(\mathfrak p,q)s_{k,r}^{\I, \J}(\mathfrak p,q)=\chi (r=n).
\end{equation}
Notice that the elliptic weights of the uncancelled cells used in \eqref{eqn:jwt} to define the 
rook numbers are different from the elliptic weights used in \eqref{eqn:filewt'},
that is, the former ones depend on the value $\J$
whereas the latter ones only depend on the coordinates of the cells,
as opposed to them being the same in the $q$-case (or $\mathfrak p,q$-case). 
Hence, the cancellation occurring in the process of proving
\eqref{eqn:Stirlingmtx} does not occur in the same way 
in the elliptic case, and as a consequence, we do not have the property that 
the matrices $||s_{n,k}^{\I, \J}(a,b;q,p)||$ 
and $||S_{n,k}^{\I, \J}(a,b;q,p)||$ are inverses of each other, 
for $s_{n,k}^{\I, \J}(a,b;q,p)=(-1)^{n-k}\mathfrak c_{n,k}^{\I, \J}(a,b;q,p)$. 
However, if we set $\I=0$ and $\J=1$, then the elliptic weights of the cells used in the rook and 
file numbers are equal, and thus we can show that the
two matrices $||s_{n,k}^{0,1}(\mathfrak p,q)||$ and
$||S_{n,k}^{0,1}(\mathfrak p,q)||$ 
are inverses of each other. 
\end{remark}


\section{Future perspectives}\label{sec:fp}


\subsection{Elliptic rook numbers on augmented boards}
In \cite{McR}, Miceli and Remmel introduced a generalized rook model by 
considering rook placements on \emph{augmented boards} and proved the
corresponding product formula as well as the $q$-analogue 
of it.
This product formula can be specialized to all the known product formulas 
including the $i$-creation model of Goldman and
Haglund~\cite{GH}. In a separate paper \cite{SY0} we provide
an elliptic extension of the $q$-case of this model.
In particular, the elliptic extension of the product formula in \cite{SY0}
includes our elliptic product formula of the $i$-creation model
as a special case.

\subsection{Elliptic rook theory for matchings}

By considering rook placements on shifted Ferrers boards subject
to a suitable modification of rook cancellation,
Haglund and Remmel~\cite{HR} developed a $q$-rook theory for
matchings of graphs. We provide an elliptic extension
of this rook theory in \cite{SY1}.
We actually consider a more general model
there related to matchings on certain graphs
which we call ``$\mathbf l$-lazy graphs'' with respect to a
$N$-dimensional vector of positive integers $\mathbf l=(l_1,\dots,l_N)$.
These matchings correspond to rook placements on $\mathbf l$-shifted
boards for which we essentially employ the same rook cancellation
as in the ordinary shifted case considered by Haglund and Remmel.
The elliptic case is more intricate than the $q$-case and
requires a very careful choice of weights.
Our factorization theorem for elliptic rook numbers for matchings
of $\mathbf l$-lazy graphs in \cite{SY1}
generalizes the factorization theorem by Haglund and Remmel
already in the ordinary case and the $q$-case.
In the simplest case, our formula can be used to
deduce an elliptic extension of the numbers of perfect matchings
of the complete graph $K_{2n}$.

\subsection{Elliptic analogue of the hit numbers}

Upon the introduction of the rook numbers, Kaplansky and Riordan~\cite{KR}
also defined the hit numbers. Let $\mathcal{H}_{n,k}(B)$ be the set of
all placements of $n$ rooks on $[n]\times [n]$ such that exactly $k$
of these rooks lie on $B$. Then $h_{n,k}(B):=|\mathcal{H}_{n,k}(B)|$ is
called the $k$-th \emph{hit number} of $B$. Kaplansky and Riordan~\cite{KR}
showed 
\begin{equation}\label{eqn:hit}
\sum_{k=0}^n r_{n-k}(B) k!(z-1)^{n-k}=\sum_{k=0}^n h_{n,k}(B)z^k.
\end{equation}
Garsia and Remmel defined in \cite{GR} the $q$-hit numbers of a Ferrers
board $B\subseteq [n]\times [n]$ algebraically by the equation 
\begin{equation}\label{eqn:qhit}
\sum_{k=0}^n r_{n-k}(q;B)[k]_q ! z^k\prod_{i=k+1}^n (1-z q^i)=
\sum_{k=0}^n h_{n,k}(q;B) z^{n-k}
\end{equation}  
and proved the existence of the statistic $\text{stat}_{n,B}(P)$ such that 
$$h_{n,k}(q;B)=\sum_{P\in \mathcal{H}_{n,k}(B)}q^{\text{stat}_{n,B}(P)}.$$
They did not give a specific description for $\text{stat}_{n,B}(P)$,
 but later, Dworkin \cite{D} and Haglund \cite{H} independently found
combinatorial descriptions. Hence the natural quest is to
establish an elliptic analogue of the relation \eqref{eqn:qhit}
and define elliptic hit numbers accordingly.

Similarly, together with the file numbers,
Garsia and Remmel also defined the fit numbers. 
Let $\mathcal{F}_{n,k}(B)$ be the set of
all file placements of $n$ rooks on $[n]\times [n]$
such that exactly $k$
of these rooks lie on $B$. Then $f_{n,k}(B):=|\mathcal{F}_{n,k}(B)|$ is
called the $k$-th \emph{fit number} of $B$.
Garsia and Remmel (cf.\ \cite{BCHR}) showed 
\begin{equation}\label{eqn:fit}
\sum_{k=0}^n f_{n-k}(B)n^k(z-1)^{n-k}=\sum_{k=0}^n f_{n,k}(B)z^k.
\end{equation}
It would be interesting to find an elliptic analogue of \eqref{eqn:fit}.
However, as a matter of fact, not even a $q$-analogue of \eqref{eqn:fit}
has so far been established. 

It should also be mentioned that Haglund and Remmel~\cite{HR} defined and
obtained results for hit and $q$-hit numbers for rook placements
on the shifted board $B_{2n}$. It would be interesting to find
extensions of their
results, either to the elliptic setting or simply to results for
$\mathbf l$-shifted boards $B_N^{\mathbf l}$ (where rook placements
are identified as partial maximal matchings on $\mathbf l$-lazy graphs)
as we considered in \cite{SY1}.

\subsection{Elliptic hypergeometric series identities}\label{subsecehs}

The coefficients $c_k(q)$ in the expansion
\begin{equation}\label{Pexp}
P(z;q)=\sum_{k=0}^n c_k(q) [z]_q [z-1]_q \cdots [z-k+1]_q
\end{equation}
are uniquely determined by $P(z;q)$.
In particular, we have (see \cite{J})
\begin{subequations}\label{eqns:cd}
\begin{equation}\label{eqn:c_k}
c_k(q) =\frac{1}{[k]_q !}\,\mathfrak z_0\,\Delta^{(k)} P(z;q)
\end{equation}
where $\mathfrak z_0$ denotes the evaluation at $z=0$ and 
\begin{equation}\label{eqn:delta_k}
\Delta^{(k)}=(\epsilon -1)(\epsilon -q)\cdots (\epsilon -q^{k-1}),
\end{equation}
\end{subequations}
with $\epsilon$ the $1$-shift operator, i.e., $\epsilon P(z;q)=P(z+1;q)$.
In fact, the formula \eqref{eqn:c_k} for the coefficients
can be simply established by expanding
\eqref{eqn:delta_k} by means of the $q$-binomial theorem
\begin{equation}
(x+y)(x+qy)\cdots (x+q^{k-1}y)=
\sum_{j=0}^k q^{\binom{j}{2}}
\begin{bmatrix}k\\j\end{bmatrix}_q y^j x^{k-j}.
\end{equation}
Taking $P(z;q)=\prod_{i=1}^n[z+b_i-i+1]_q$, where $B=B(b_1,\dots,b_n)\subseteq
[n]\times\mathbb N$ is a Ferrers board, the coefficients in \eqref{Pexp}
become $c_k(q)=r_{n-k}(q;B)$, due to Proposition~\ref{thm:qrookthm}.
Thus one has
\begin{equation}\label{explrook}
r_{n-k}(q;B)=\frac{1}{[k]_q !}\,\mathfrak z_0\,
\Delta^{(k)}\prod_{i=1}^n[z+b_i-i+1]_q.
\end{equation}
In \cite[Equations (24) and (49)]{H0}, Haglund has used
an inversion argument involving the $q$-Chu--Vandermonde
summation to arrive at an identity equivalent to \eqref{explrook}.
For ``regular'' boards (these are Ferrers boards whose shape correspond
to Dyck paths), he was then able to express $q$-rook numbers
in terms of multiples of basic hypergeometric series of
Karlsson--Minton type (see \cite{GRhyp} for the terminology).
By analytic continuation such formulas can also be obtained
for some other boards which are not ``regular''
(such as the staircase board $\mathsf{St}_n$ whose
$q$-rook numbers are Stirling numbers of the second kind
which also admit a basic hypergeometric series representation of
Karlsson--Minton type, see \eqref{stirlingkmtype}). 
Haglund further combined his results with identities
involving $q$-hit numbers to obtain transformation
formulae for basic hypergeometric series of Karlsson--Minton type.

Now it would be very nice if one could extend the analysis
we described leading to \eqref{explrook}
to the elliptic setting, given that we were able to establish an
elliptic analogue of the product formula in Proposition~\ref{thm:qrookthm},
namely Theorem~\ref{thm:elptprod}, and also given that elliptic analogues
of basic hypergeometric series of Karlsson--Minton type exist
(see \cite{RS} or \cite[Chapter~11]{GRhyp}).
At first glance, the elliptic extension of this approach
may appear straightforward. What is needed is an explicit formula
for the coefficients $c_k(a,b;q,p)$ in
\begin{equation}
P(z;a,b;q,p)=\sum_{k=0}^n c_k(a,b;q,p) [z]_{a,b;q,p} [z-1]_{aq^2,bq;q,p}
\cdots [z-k+1]_{aq^{2(k-1)},bq^{k-1};q,p},
\end{equation}
and then take
\begin{equation*}
P(z;a,b;q,p)=\prod_{i=1}^n[z+b_i-i+1]_{aq{2(i-1-b_i)},bq^{i-1-b_i};q,p}.
\end{equation*}
Although we so far failed
to find a correct elliptic extension of \eqref{eqns:cd},
a solution of this problem does not seem to be
completely out of reach.

For the reader to get a feeling for some of the difficulties
in the elliptic setting, compare Carlitz' compact formula for
the $q$-Stirling numbers of the second kind in \eqref{carlitzexpl}
(which, as mentioned, can also be written in terms of
basic hypergeometric series of Karlsson--Minton type,
see \eqref{stirlingkmtype})
with the (not so uniform) expressions for the elliptic
Stirling numbers of the second kind which we listed in \eqref{snk}.
The latter appear not to be instances of a multiple of an
elliptic hypergeometric series of Karlsson--Minton type.
Nevertheless they may be instances of a series which is
close to being elliptic hypergeometric\footnote{Following
\cite[Chapter 11]{GRhyp}, a series
$\sum_{k\ge 0}c_k$ is called \emph{elliptic hypergeometric} if and only if
the quotient $c_{k+1}/c_k$ is an elliptic (i.e., meromorphic and
doubly periodic) function in $k$, where $k$ is viewed
as a complex variable.}.

\subsection{Relations to algebraic varieties}

Ding introduced the \emph{length function} which was first used as
the length of rook matrices by Solomon in his work on the Iwahori
ring of $M_n (F_q)$ \cite{Sol}. Given a Ferrers board $B=B(b_1,\dots, b_n)$
and a $k$-rook placement $P\in \mathcal{N}_k (B)$, the length function
$l_B(P)$ is defined by 
$$l_B (P)=\sum_{(i,j)\in P}(n-i+j-1+\gamma_{(i,j)}),$$
where $\gamma_{(i,j)}$ is the number of rooks which are in the
south-east region of the rook in $(i,j)$. Then, if we let
$C_{B,k}=\sum_{i=1}^n b_i -\frac{k(k+1)}{2}$, we have 
$$u_B(P)+l_B (P)=C_{B,k}.$$
Thus, if we define the rook length polynomial by
$$rl_k (q;B)=\sum_{P\in\mathcal{N}_{k}(B)}q^{l_B (P)},$$
then the relation to the $q$-rook number is 
$$r_k (q;B)=q^{C_{B,k}}rl_k (q^{-1};B).$$
In \cite{Ding1}, Ding studied the geometric implication of rook length
polynomials by introducing partition varieties. Partition varieties are
projective varieties which have cellular decomposition analogous to the
cellular decomposition of the Grassmannian into Schubert cells.
These partition varieties have CW-complex structure with the rook length
polynomials being their Poincar\'{e} polynomials of cohomology.
In \cite{Ding2}, Ding even generalized the study of partition varieties
by replacing the Borel subgroup of upper triangular matrices by
more general parabolic subgroups of the general linear group.
For this purpose he introduced $\gamma$-compatible partitions,
$\gamma$-compatible rook placements
and $\gamma$-compatible rook length polynomials, where $\gamma$ is a
composition. 

It would be interesting to reveal any connection between
elliptic analogues of rook numbers and algebraic varieties.




\begin{thebibliography}{99}

\bibitem{BLL} F.~Bergeron, G.~Labelle and P.~Leroux,
``Combinatorial Species and Tree-Like Structures", 
{\em Encyclopedia of Mathematics and Its Applications}, \text{Vol.\ 67},
Cambridges Univ.\ Press, Cambridge, UK, 1998.

\bibitem{Be} D.~Betea,
``Elliptic combinatorics and Markov processes'',
{Ph.D.\ thesis}, California Institute of Technology, 2012.

\bibitem{BR} K.~S.~Briggs and J.~B.~Remmel,
``A $p,q$-analogue of a formula of {F}robenius",
{\em Electron.\ J. \ Combin.} \textbf{10} (2003), \#R9.

\bibitem{B} A.~Broder,
``The $r$-Stirling numbers",
{\em Discrete Math.\ }\textbf{49} (1984), 241--259.

\bibitem{BCHR} F.~Butler, M.~Can, J.~Haglund and J.~B.~Remmel,
``Rook theory notes'', book project
\texttt{http://www.math.ucsd.edu/{\textasciitilde}remmel/files/Book.pdf}

\bibitem{C1} L.\ Carlitz,
``On abelian fields'',
{\em Trans.\ Amer.\ Math.\ Soc.\ }\textbf{35} (1933), 122--136.

\bibitem{C2} L.\ Carlitz,
``$q$-Bernoulli numbers and polynomials'',
{\em Duke Math.\ J.\ }\textbf{15} (1948), 987--1000.

\bibitem{BGR} A.~Borodin, V.~Gorin, E.M.~Rains,
``$q$-Distributions on boxed plane partitions'',
{\em Sel.\ Math.\ (N.S.)} \textbf{16} (2010), no.~4, 731--789.

\bibitem{DJKMO} E.~Date, M.~Jimbo, A.~Kuniba, T.~Miwa, M.~Okado,
``Exactly solvable SOS models: local height probabilities and
theta function identities'', {\em Nuclear Phys.\ B}
\textbf{290} (1978), 231--273.

\bibitem{ML} A.~de~M\'{e}dicis and P.~Leroux,
``Generalized Stirling numbers, convolution formulae and 
$p$, $q$-analogues",
{\em Can. J.\ Math.\ }\textbf{47} (1995), 474--499.

\bibitem{Ding1} K.\ Ding,
``Rook placements and cellular decomposition of partition varieties",
{\em Discrete Math.\ }\textbf{170} (1997), 107--151.

\bibitem{Ding2} K.\ Ding, 
``Rook placements and generalized partition varieties",
{\em Discrete Math.\ }\textbf{176} (1997), 63--95.

\bibitem{D} M.\ Dworkin,
``An interpretation for Garsia and Remmel's $q$-hit numbers",
{\em J.\ Combin.\ Theory Ser.\ A} \textbf{81} (1998), 149--175.

\bibitem{F} D.~Foata,
``Distribution eul\'eriennes et mahoniennes sur le groupe des permutations'',
With a comment by Richard P.~Stanley,
{\em NATO Adv.\ Study Inst.\ Ser.\ Ser.\ C: Math.\ Phys.\ Sci.\ }\textbf{31},
Higher combinatorics (Proc.\ NATO Advanced Study Inst., Berlin, 1976),
pp.~27--49; Reidel, Dordrecht/Boston, MA, 1977.

\bibitem{FT}
I.~B.\ Frenkel and V.~G.\ Turaev,
``Elliptic solutions of the
{Y}ang--{B}axter equation and modular hypergeometric functions'',
in: V.I.\ Arnold, I.M.\ Gelfand, V.S.\ Retakh and M.\ Smirnov (eds.),
\textit{The {A}rnold--{G}elfand mathematical seminars}, 
Birkh\"auser, Boston, 1997, pp.~171--204.

\bibitem{GR0} A.~M.~Garsia and J.~B.~Remmel,
``A combinatorial interpretation of the $q$-derangement and
$q$-Laguerre numbers'',
{\em European J.\ Combin.\ }\textbf{1} (1980), 47--59.

\bibitem{GR} A.~M.~Garsia and J.~B.~Remmel,
``$Q$-counting rook configurations and a formula of Frobenius'',
{\em J.\ Combin.\ Theory Ser.\ A} \textbf{41} (1986), 246--275.

\bibitem{GRhyp} G.~Gasper and M.~Rahman,
{\em Basic hypergeometric series}, second edition,
Encyclopedia of Mathematics and Its Applications~\textbf{96},
Cambridge University Press, Cambridge, 2004.

\bibitem{GJW} J.~R.~Goldman, J.~T.~Joichi and D.~E.~White,
``Rook Theory I, Rook equivalence of Ferrers boards",
{\em Proc.\ Amer.\ Math.\ Soc.\ }\textbf{52} (1975), 485--492.

\bibitem{GH} J.~Goldman and J.~Haglund,
``Generalized rook polynomials",
{\em J.\ Combin.\ Theory Ser.\ A}\textbf{91} (2000), 509--530.

\bibitem{H0} J.~Haglund,
``Rook theory and hypergeometric series",
{\em Adv.\ in Appl.\ Math.\ }\textbf{17} (1996), 408--459.

\bibitem{H} J.~Haglund,
``$q$-Rook polynomials and matrices over finite fields",
{\em Adv.\ in Appl.\ Math.\ }\textbf{20} (1998), 450--487.

\bibitem{HR} J.~Haglund and J.~B.~Remmel,
``Rook theory for perfect matchings",
{\em Adv.\ in Appl.\ Math.\ }\textbf{27} (2001), 438--481.

\bibitem{J} F.~H.~Jackson,
``$q$-Difference equations'',
{\em Amer.\ J.\ Math.\ }\textbf{32} (1910), 305--314.

\bibitem{Ji} M.~Jimbo,
``Introduction to the Yang--Baxter equation'',
{\em Int.\ J.\ Modern Phys.\ A} \textbf{4} (15) (1989), 3759--3777.

\bibitem{KR} I.~Kaplansky and J.~Riordan,
``The problem of the rooks and its applications'',
{\em Duke Math.\ J.\ }\textbf{13} (1946), 259--268.

\bibitem{KBSchl} Z.~R.~Keresk\'enyin\'e Balogh and M.~J.~Schlosser,
``Elliptic Stirling numbers of the second and first kind'',
in preparation.

\bibitem{MS} T.~Mansour and M.~Schork,
``Commutation Relations, Normal Ordering, and Stirling Numbers'',
Chapman and Hall/CRC Press, 2015.

\bibitem{McR} B.~K.~Miceli and J.\ Remmel,
``Augmented rook boards and general product formulas'',
{\em Electronic J.\ Combin.\ }\textbf{15} (2008), \#R85.

\bibitem{MR} R.~Mullin and G.~C.~Rota,
``On the foundations of combinatorial theory. III.
Theory of binomial enumeration", in
{\em Graph Theory and Its Applications}, (B.~Harris,~Ed.), 167--213,
Academic Press, New York (1970);
reprinted in ``Gian-Carlo Rota on Combinatorics,
Introductory Papers and Commentaries" (Joseph~P.~S.~Kung,~Ed.),
118--147, Birkhauser, 1995.

\bibitem{NR} G.~Nyul and G.~R\'acz,
``The $r$-Lah numbers'',
{\em Discrete M.\ }\textbf{338} (2015), 1660--1666.

\bibitem{RW} J.~B.~Remmel and M.~Wachs,
``Rook theory, generalized Stirling numbers and $(p,q)$-analogues'',
{\em Electronic J.\ Combin.\ }\textbf{11} (2004), \#R84.

\bibitem{RS} H.~Rosengren and  M.~J.~Schlosser,
``On Warnaar's elliptic matrix inversion and Karlsson--Minton-type
elliptic hypergeometric series''.
{\em J.\ Comput.\ Appl.\ Math.\ }\textbf{178} (2005), 377--391.

\bibitem{Schl0}  M.~J.~Schlosser,
``Elliptic enumeration of nonintersecting lattice paths'',
{\em J.\ Combin.\ Theory Ser.\ A} \textbf{114} (2007), 505--521.

\bibitem{Schl1}  M.~J.~Schlosser,
``A noncommutative weight-dependent generalization of the binomial theorem'',
preprint \texttt{arXiv:1106.2112}.

\bibitem{SY0}  M.~J.~Schlosser and M.~Yoo,
``An elliptic extension of the general product formula for augmented
rook boards'', Europ.\ J.\ Combin., to appear;
preprint \texttt{arXiv:1601.07834}.

\bibitem{SY1}  M.~J.~Schlosser and M.~Yoo,
``Elliptic extensions of the alpha-parameter model
and the rook model for matchings'', preprint.

\bibitem{Sl} N.~J.~A.~Sloane,
``The On-Line Encyclopedia of Integer Sequences'',
published electronically at \texttt{http://oeis.org}, 2010.

\bibitem{Sol} L.~Solomon,
``The Bruhat decomposition, Tits system and Iwahori ring
for the monoid of matrices over a finite field",
{\em Geom.\ Dedicata.\ }\textbf{36} (1990), 15--49.

\bibitem{Sp} V.~P.~Spiridonov, 
\textit{Theta hypergeometric series},
in V.~A.~Malyshev and A.~M.~Vershik (eds.),
\textit{Asymptotic Combinatorics with Applications to Mathematical Physics},
Kluwer Acad.\ Publ., Dordrecht, 2002, pp.~307--327.

\bibitem{SZ}
V.~P.~Spiridonov and A.~S.~Zhedanov.
``Spectral transformation chains and some new biorthogonal
rational functions'',
{\em Comm.\ Math.\ Phys.}\textbf{10} (2000), 49--83.

\bibitem{Stan} R.~P.~Stanley, ``Enumerative Combinatorics", vol. 1,
Cambridge University Press, 
Cambridge, (1997).

\bibitem{Wa}
S.O.\ Warnaar,
``Summation and transformation formulas for elliptic
hypergeometric series'',
{\textit{Constr.\ Approx.\ }}\textbf{18} (4) (2002), 479--502.

\bibitem{WaW} M.~Wachs and D.~White,
``$p,q$-{S}tirling numbers and set partition statistics",
{\em J.\ Combin.\ Theory Ser.\ A} \textbf{56}, (1991), 27--46.

\bibitem{W} H.\ Weber,
\textit{Elliptische Functionen und Algebraische Zahlen}, Vieweg-Verlag,
Braunschweig, 1897.

\bibitem{WhW} E.~T.~Whittaker and G.~N.~Watson, {\em A Course of
Modern Analysis}, 4th ed., Cambridge University Press, Cambridge, 1962.

\end{thebibliography}
\end{document}